\documentclass{amsart}
\usepackage[normalem]{ulem}
\usepackage{amssymb,epsfig,mathrsfs,mathpazo,stackengine,scalerel,url,amsthm,thmtools}
\usepackage{xcolor}
\usepackage{stmaryrd}
\usepackage{epsfig,bm}
\usepackage[multiple]{footmisc}
\usepackage{accents} 
\usepackage{mathtools} 
\usepackage{dsfont}
\usepackage{fancyvrb}
\usepackage{arydshln, booktabs, makecell, pbox, tabularx}

\usepackage{graphicx, caption, subcaption}

\newtheorem{theorem}{Theorem}[section]

\newtheorem{proposition}[theorem]{Proposition}

\newtheorem{assumption}[theorem]{Assumption}

\declaretheorem[style=definition,qed=$\vartriangle$,sibling=theorem]{example}
\declaretheorem[style=remark,qed=$\vartriangle$,sibling=theorem]{remark}
\numberwithin{equation}{section}

\newcommand{\eps}{\varepsilon}

\newcommand{\R}{\mathbb R}

\newcommand{\N}{\mathbb N}

\newcommand{\cF}{\mathcal F}

\newcommand{\cE}{\mathcal E}

\newcommand{\cS}{\mathcal S}
\newcommand{\cL}{\mathcal L}
\newcommand{\Lis}{\cL\mathrm{is}}

\newcommand{\identity}{\mathrm{Id}}

\DeclareMathOperator{\ran}{ran}
\DeclareMathOperator{\dom}{dom}
\DeclareMathOperator{\supp}{supp}

\DeclareMathOperator{\diam}{diam}
\DeclareMathOperator*{\argmin}{argmin}
\DeclareMathOperator{\dist}{dist}

\DeclareMathOperator{\divv}{div}

\DeclareMathOperator{\Span}{span}

\newcommand{\nrm}{| \! | \! |}

\newcommand{\corr}[1]{#1}

\newcommand{\be}{\begin{equation}}
\newcommand{\ee}{\end{equation}}

\newcommand{\tria}{{\mathcal T}}
\newcommand{\cP}{{\mathcal P}}

\newcommand{\uumlaut}{{\"u}}
\newcommand{\oumlaut}{{\"o}}

\newcommand{\osc}{{\rm osc}}

\stackMath
\newcommand\tenq[2][1]{%
 \def\useanchorwidth{T}%
  \ifnum#1>1%
    \stackunder[0pt]{\tenq[\numexpr#1-1\relax]{#2}}{\scriptscriptstyle\sim}%
  \else%
    \stackunder[1pt]{#2}{\scriptscriptstyle\sim}%
  \fi%
}

\title{Least squares solvers for ill-posed PDEs that are conditionally stable}
\date{\today}

\author{Wolfgang Dahmen}
\address{Mathematics Department, University of South Carolina, Columbia SC 29208}
\email{wolfgang.anton.dahmen@googlemail.com}
\author{Harald Monsuur}
\author{Rob Stevenson}
\address{Korteweg-de Vries (KdV) Institute for Mathematics, University of Amsterdam, P.O. Box 94248, 1090 GE Amsterdam, The Netherlands.}
\email{h.monsuur@uva.nl, rob.p.stevenson@gmail.com}

\thanks{This research has been supported in part by the NSF Grant DMS  ID 2012469, by the SmartState and Williams-Hedberg Foundation, by the SFB 1481, funded by the 
German Research Foundation, and by the Netherlands Organization for Scientific Research (NWO) under contract.~no.~SH-208-11. 
We acknowledge the support of SURF (www.surf.nl) in using the National Supercomputer Snellius.}



\begin{document}

\begin{abstract} 
This paper is concerned with the design and analysis of least squares solvers for ill-posed PDEs that are conditionally stable.
The norms and the regularization term
used in the least squares functional are determined by the ingredients of
the conditional stability assumption. 
We are then able to establish a general
error bound that, in view of the conditional stability assumption, is qualitatively the best possible, without assuming consistent data. The price for these advantages is to handle 
dual norms which reduces to verifying suitable inf-sup stability.
This, in turn, is done by constructing appropriate Fortin projectors
for all sample scenarios. The theoretical findings are illustrated by numerical experiments.
 \end{abstract}
 

\maketitle

\section{Introduction}
In this paper a general approach is developed for the numerical solution of ill-posed  boundary value problems $A u=f$, where $A \in \cL(X,V)$ for Hilbert spaces $X$ and $V$,
that are \emph{conditionally stable}.
The latter means that for all $u \in X$ that satisfy an a priori bound of the form $\|L u\|_H \leq C$, for some $L \in \cL(X,H)$ and a Hilbert space $H$,
there is a continuous dependency of the solution on the data in the sense that, 
for some $\eta=\eta_C\colon \R^+ \rightarrow \R^+$ with $\lim_{t \downarrow 0} \eta(t)=0$, 
it holds that $j(u) \leq \eta \big(\|A u\|_V\big)$ for some $j\colon X \rightarrow \R_+$.
In applications $j(\cdot)$ is a (semi-) norm that is weaker than 
the norm on $X$, and, e.g., $\eta(t)=t^s$ or even only $\eta(t)=(-1/\log t)^s$ for some
$s \in (0,1]$.
Conditional stability has been established for various ill-posed PDEs including data-assimilation and Cauchy boundary data problems for Poisson's, heat and wave equations.

For finite dimensional subspaces $X^\delta \subset X$ (`$\delta$' refers to `discrete'), we approximate $u$ by the minimizer $u_\eps^\delta$ over $X^\delta$
of the regularized least-squares functional $z \mapsto \sqrt{\|Az-f\|_V^2+\eps^2 \|Lz\|_H^2}$.
For a suitable selection of $\eps$, it will be shown that both $\|L(u-u_\eps^\delta)\|_H$ is uniformly bounded, so that 
$j(u-u_\eps^\delta) \leq \eta \big(\|f-A u_\eps^\delta\|_V\big)$, and
$\|f-A u_\eps^\delta\|_V$ is bounded by an absolute multiple of 
$\|f-Au\|_V +\min_{z \in X^\delta}\|u-z\|_X$, being the sum of the consistency error and the error of best approximation.
Consequently, we will achieve qualitatively the best possible bound on the error quantity $j(u-u_\eps^\delta)$ that can be expected for  $u_\eps^\delta \in X^\delta$ in view of the conditional stability estimate.

In applications often $V$ is of product form $\prod_i V_i$, and so $A=(A_i)$ and $f=(f_i)$, 
with one or more $V_i$ being a Sobolev space with negative smoothness index, which is a natural space for a forcing term of a PDE, 
or a fractional Sobolev space on (a part of) the boundary of the computational domain, which is a natural space for a boundary datum. The norms on such spaces cannot be evaluated exactly.

An option to deal with a Sobolev norm of negative smoothness index is to replace it by an $L_2$-norm.
This, however, requires more smoothness of the data and more regularity of $X^\delta$, e.g., a $C^1$- instead of a $C^0$-finite element space,
whereas it is not ensured that the error benefits from smallness of the residual in a stronger norm.
Similar disadvantages are connected to the replacement of fractional Sobolev norms by (weighted) $L_2$-norms.

Our approach to deal with $V_i$ being a Sobolev space with negative smoothness index, i.e., a $V_i$ being of the form $Y_i'$,
is to replace $\|\cdot\|_{Y_i'}$ in the least-squares functional by a discrete dual norm $\|\cdot\|_{{Y^\delta_i}'}$, where $Y_i^\delta=Y_i^\delta(X^\delta) \subset Y_i$
is such that $\|A_i\cdot\|_{Y_i'}$ is equivalent to $\|A_i\cdot\|_{{Y^\delta_i}'}$ on $X^\delta$, and $\dim Y_i^\delta$ is proportional to $\dim X^\delta$. 
The first property is known to be equivalent to existence of a (uniformly bounded) Fortin projector $Y_i \rightarrow Y_i^\delta$.

By introducing the Riesz lift of the corresponding residual $f_i-A_i u^\delta_\eps \in {Y_i^\delta}'$ as an independent variable, the resulting least squares problem has an equivalent formulation as a mixed system, 
which does not involve the dual norm $\|\cdot\|_{{Y^\delta_i}'}$, and which is Ladyshenskaja-Babu\u{s}ka-Brezzi (LBB) stable by virtue of the existence of the Fortin projector.
In many cases, one can construct a $G^\delta_{Y_i}\colon {Y^\delta_i}' \rightarrow Y^\delta_i$ (known as a preconditioner) with $({G^\delta_{Y_i}}^{-1} v)(v)$ equivalent to $\|v\|_{Y_i}^2$, and whose application can be performed in linear complexity, in which case one can efficiently eliminate the additional variable and so retrieves a symmetric positive definite system.

We handle fractional Sobolev norms in the same manner. Viewing a fractional Sobolev space $V_i$, with either a positive or negative smoothness index,
as the dual of $Y_i:=V_i'$, first we construct a (uniformly bounded) Fortin projector $Y_i \rightarrow Y_i^\delta$, and second, to avoid having to compute fractional norms with opposite index of arguments from $Y_i^\delta$, we use a preconditioner $G^\delta_{Y_i}\colon {Y^\delta_i}' \rightarrow Y^\delta_i$ with $({G^\delta_{Y_i}}^{-1} v)(v)$ equivalent to $\|v\|_{Y_i}^2$.

The  steps to handle dual or fractional norms, mentioned above, make our approach practically feasible without compromizing its attractive theoretical properties.
Indeed, still one obtains a bound on the error quantity $j(u-u_\eps^\delta)$ that is qualitatively the best possible.

We exemplify our approach by constructing Fortin interpolators and preconditioners for the examples of the Cauchy problem for Poisson's equation, and data-assimilation problems for wave- and heat-equations. Furthermore, for those examples we illustrate our theoretical findings with numerical results.
\medskip

Our approach to minimize a regularized least squares functional is of course not new.
Not making use of conditional stability, in \cite{19.897,28.5,35.828} this method was analyzed for a regularizing term $\eps^2 \|z\|_X^2$, instead of  our choice  $\eps^2 \|Lz\|_H^2$ suggested by the conditional stability condition.
By replacing the test function from $X$ by minus this test function, a non-symmetric mixed system on $X \times V$ is obtained that is coercive, with a coercivity constant that is, however, proportional to $\eps^2$. With this formulation the notion of stability is fully due to the stabilization term.
By our approach to guarantee LBB-stability, 
the operator $A$ contributes to the stability of the least-squares problem. 
It results in a proof of convergence rates in the error quantity associated to the conditional stability estimate that seems new.

The use of conditional stability estimates for the numerical solution of various ill-posed PDEs has been advocated in series of papers \cite{35.8585,35.859,35.8595, 35.925,35.926,35.928,35.929,35.9296,35.9297}. In those works a control functional is minimized under the constraint that the state satisfies the PDE. Instead of adding Tikhonov stabilization at the continuous level, mesh-dependent stabilization terms tailored to the application at hand are added to the finite element discretization.

\subsection{Layout}\label{ssec:layout}  
In Sect.~\ref{sec:2} we recall the concept of conditional stability for ill--posed problems $Au=f$.
Under the provision that approximations $u_\eps^\delta$ to $u$ from finite dimensional spaces $X^\delta$ are available that satisfy a certain quasi-optimal error bound in an $\eps$-dependent energy norm, it will be demonstrated that, for a judiciously chosen regularization parameter $\eps$,
a qualitatively best possible upper bound holds for the error quantity $j(u-u_\eps^\delta)$.
In Sect.~\ref{Sappls} we present several classical examples to which the theory
applies.
In Sect.~\ref{LS} the aforementioned quasi-optimal error bound will be demonstrated for $u_\eps^\delta \in X^\delta$ being the minimizer of a regularized least squares functional, in which, under inf-sup conditions, dual norms are replaced by discrete dual norms.
The resulting least-squares problem has an equivalent formulation as a mixed system.
Sect.~\ref{Sposdef} is devoted
to a reformulation of the mixed problem as a symmetric positive
definite variational problem, based on uniform preconditioners that
serve as approximate Riesz lifters in those Hilbert space components
that require the use of dual norms. In Sect.~\ref{Sapost} we discuss
a posteriori residual estimators.
In Sect.~\ref{sec:inf-sup} we verify the validity of
the critical inf-sup conditions for all sample problems, the results in preceding sections hinge upon.
Here the central work horse are suitable Fortin operators.
Finally, in Sect.~\ref{sec:numer} we present numerical results for our three sample problems.

\subsection{Notation}\label{ssec:1.2}
In this work, by $C \lesssim D$ we will mean that $C$ can be bounded by a multiple of $D$, \emph{independently} of parameters which $C$ and $D$ may depend on, as the discretisation index $\delta$, the tolerance $\tau$ for the consistency error, and the regularization parameters $\eps$ and $\zeta$.
Obviously, $C \gtrsim D$ is defined as $D \lesssim C$, and $C\eqsim D$ as $C\lesssim D$ and $C \gtrsim D$.

For normed linear spaces $E$ and $F$, by $\cL(E,F)$ we will denote the normed linear space of bounded linear mappings $E \rightarrow F$,
and by $\Lis(E,F)$ its subset of boundedly invertible linear mappings $E \rightarrow F$.
We write $E \hookrightarrow F$ to denote that $E$ is continuously embedded into $F$.
For  convenience only, we exclusively consider linear spaces over the scalar field $\R$.

The set $[0,\infty)$ will be denoted by $\R^+$.

\section{Problem setting and main result} \label{sec:2}
 For Hilbert spaces $X$ and $V$, we consider operators $A \in \cL(X,V)$ which are neither assumed to be injective nor to have a dense range in $V$.
We study the problem of the (approximate) reconstruction of $u \in X$ from its image $Au$ assuming we are only given a perturbation $f$ of $Au$ for which 
\be 
\label{problem}
\|f-Au\|_V \leq \tau
\ee
holds for some tolerance $\tau\geq 0$ which we assume to be known.  
 Since in particular we do not assume bounded invertibility of $A$ our problem is ill-posed, even for $\tau=0$. 
 Although for convenience  we refer in the following to $u$ as \emph{the} solution 
of our recovery problem, one should bear  in mind that   for $\tau>0$ there may  be multiple $u \in X$ that satisfy \eqref{problem}. Our results will be valid uniformly in those $u$.

Since Tikhonov (\cite{250.2}) with \emph{stability} for ill-posed problems, usually called \emph{conditional stability}, one understands \emph{some} continuous dependency of the solution upon the data,  typically with respect to a weaker metric than that induced by $\|\cdot\|_X$, under the assumption that a  bound on the solution itself is available.
More specifically, similar to \cite{35.8585}, we assume existence of an $L \in  \cL(X,H)$, where $H$ is some additional Hilbert space, such that
 the following assumption is valid:
\begin{assumption}[Conditional stability] \label{assump} The pair
$$
(A,L)\in \cL(X,V \times H) \,\,\text{ is injective,}
$$
and there exists a $j\colon X \rightarrow \R^+$, and for any $\mathcal{C}>0$, a non-decreasing $\eta=\eta_{\mathcal{C}}\colon \R^+ \rightarrow \R^+$ with $\lim_{t \downarrow 0} \eta(t)=0$, such that for $z\in X$ with $\|Lz\|_H \leq {\mathcal{C}}$, it holds that
\be \label{jz}
j(z) \leq \eta_{\mathcal{C}}\big(\|A z\|_V\big).
\ee
\end{assumption}

\noindent Typically, $j$ is a norm or a semi-norm on a Hilbert space $\widetilde{H} \hookleftarrow X$. In the first case, $A$ is injective.
Several examples will be given in Sect.~\ref{Sappls}.

For $\eps > 0$, we set
\be \label{trip}
\nrm\cdot\nrm_\eps:=\sqrt{\|A \cdot\|_V^2+\eps^2\|L\cdot\|_H^2}.
\ee
To use precisely the ingredients of the conditional stability 
condition in the definition of $\nrm\cdot\nrm_\eps$ will be seen to be essential in
what follows.
Moreover, notice that $\nrm\cdot\nrm_\eps$ is a norm on $X$
 by our assumption of $(A,L)$ being injective. 
Thinking of $\eps$ being small,  tacitly we will always assume that $\eps \|L\|_{\cL(X,H)} \lesssim 1$ so that, since $A$ is bounded,  $\nrm\cdot\nrm_\eps \lesssim \|\cdot\|_X$.

Given a \emph{finite dimensional} subspace $X^\delta$ of $X$, in Sect.~\ref{LS}-\ref{Sposdef} we show how to compute for each $\eps>0$ a $u_\eps^\delta \in X^\delta$ satisfying  
\be \label{1}
\nrm u -u_\eps^\delta \nrm_\eps \lesssim \tau+\min_{z \in X^\delta}\nrm u-z\nrm_\eps+\eps \|Lu\|_H.
\ee

If \eqref{jz} is valid for $L=0$, and thus $\eta_{\mathcal{C}} \equiv \eta$ is independent of ${\mathcal{C}}$, then one speaks about \emph{unconditional stability} of \eqref{problem}.\footnote{Not to be confused with \emph{well-posedness}, with which we mean $A\in \Lis(X,V)$.
An unconditionally stable problem where $\ran A$ is closed is also benign in the sense that then, by an application of the open mapping theorem, \eqref{problem} is \emph{well-posed in least-squares sense}, i.e., $A^* A \in \Lis(X,X)$, and thus is also not of our primary interest.}
In this case $\nrm\cdot\nrm_\eps=\|A\cdot\|_V$ is $\eps$-independent, and so will be $u_\eps^\delta$.

\begin{theorem} \label{thm:0} Assume \eqref{1} and recall Assumption~\ref{assump}.
For $L\neq 0$, let $\eps=\eps(\tau,\delta)>0$ be such that for some ${\mathcal{C}} \geq \|Lu\|_H$,
 \be \label{15}
 \tau+\min_{z \in X^\delta}\nrm u-z\nrm_\eps \lesssim \eps {\mathcal{C}} \quad \text{ and } \quad \eps \|L u\|_{H} \lesssim  \tau+\min_{z \in X^\delta}\| u-z\|_X.
  \ee
 Then
 \be \label{25}
j(u-u_\eps^\delta) \leq \eta_{\mathcal{C}}\big(\|A (u-u_\eps^\delta)\|_V\big),
\ee
and
\be \label{9}
\|A (u-u_\eps^\delta)\|_V
\lesssim  \tau+\min_{z \in X^\delta}\| u-z\|_X.
\ee
\end{theorem}
  
 \begin{proof} When $L=0$, \eqref{25} holds unconditionally, and $\nrm \cdot\nrm_\eps = \|A \cdot\|_V\lesssim \|\cdot\|_X$ so that \eqref{9} follows directly from \eqref{1}.
So let $L\neq 0$.
From \eqref{1} and the lower bound from \eqref{15} on $\eps$, one 
derives $\|L(u-u_\eps^\delta)\|_H \leq \eps^{-1} \nrm u -u_\eps^\delta \nrm_\eps \lesssim {\mathcal{C}}$. Conditional stability then ensures
$$
j(u-u_\eps^\delta) \leq \eta_{\mathcal{C}}\big(\|A (u-u_\eps^\delta)\|_V\big).\footnotemark
$$
\footnotetext{Although in applications, the decay of $\eta_{\mathcal{C}}(t)$ for $t \downarrow 0$ is faster when simultaneously ${\mathcal{C}}={\mathcal{C}}(t) \downarrow 0$, by the very character of a conditional stability estimate $\|Lz\|_H \downarrow 0$ cannot be expected to follow from $\|A z\|_V \downarrow 0$.}%
Again \eqref{1}, and the upper bound on $\eps$ from  \eqref{15} show that
$$
\|A (u-u_\eps^\delta)\|_V \leq \nrm u -u_\eps^\delta \nrm_\eps 
\lesssim \tau+\min_{z \in X^\delta}\| u-z\|_X,
$$
which completes the proof.\qedhere
\end{proof}

Notice  that when $\min_{z \in X^\delta}\| u-z\|_X$ in \eqref{15} is replaced by an upper bound, then one arrives at \eqref{9} with $\min_{z \in X^\delta}\| u-z\|_X$ replaced by that upper bound.

In the examples given in Sect.~\ref{Sappls}, $\eta_{\mathcal{C}}(t)$ will be of the form ${\mathcal O}((t+\mathcal{C})^{1-\sigma} t^\sigma)$
or ${\mathcal O}((t+\mathcal{C}) (\log(1+\frac{{\mathcal{C}}}{t}))^{-\sigma})$
 for some $\sigma \in (0,1)$, or ${\mathcal O}(t)$.

\begin{remark}[{Optimality of estimates}] 
\label{remmie} The bound \eqref{9} on $\|A (u-u_\eps^\delta)\|_V$, (valid because of the upper bound on $\eps$),  by a multiple of the sum of the {(maximal) \emph{consistency error} $\tau$} and the \emph{approximation error} $\min_{z \in X^\delta}\|u-z\|_X$, is qualitatively the best that can be expected for a numerical approximation from $X^\delta$.
Since, thanks to the lower bound on $\eps$, at the same time $\|L(u-u_\eps^\delta)\|_H \lesssim {\mathcal{C}}$,
we obtain the generally qualitatively best possible upper bound for $j(u-u_\eps^\delta)$ that is permitted by the conditional stability estimate.

That being said, inserting the upper bound in \eqref{9} into \eqref{25} can nevertheless provide a pessimistic bound. The reason is that perturbations in the data enter the approximation by solving a discretized problem, whose conditioning is for coarser and coarser meshes usually increasingly better than that of the infinite dimensional problem whose behaviour is captured by the conditional stability estimate.
\end{remark}

\begin{remark}[Selection of $\eps$ when $L \neq 0$] \label{rem:graph} 
Let $(X^\delta)$ be a family of finite dimensional subspaces of $X$ such that for some $s>0$ for general, sufficiently smooth $u \in X$ it holds $\min_{z \in X^\delta}\| u-z\|_X \eqsim (\dim X^\delta)^{-s}$. Then, in view of \eqref{15}, an obvious choice is to take $\eps \eqsim \tau + (\dim X^\delta)^{-s}$. Because of a lacking smoothness of $u$, it might be, however,  that this $\eps$ decays too fast for $\dim X^\delta \rightarrow \infty$ which then would manifest itself by an increase of $\|L u_\eps^\delta\|_H$.
Indeed, recall that the sole reason for imposing the lower bound on $\eps$ in \eqref{15} is to prevent an unbounded growth of $\|L u_\eps^\delta\|_H$, and thus of $\|L(u-u_\eps^\delta)\|_H$, which would jeopardize a meaningful application of the conditional stability estimate.
The value of $\|L u_\eps^\delta\|_H$, however, can be monitored and so the choice of $\eps$ can be adapted when such a growth of $\|L u_\eps^\delta\|_H$ is observed.

In various numerical experiments we observed that regularization is actually not needed at all, whereas for other data $\eps$ equal to $\tau$ was close to the experimentally found best regularization parameter.  In our tests, where $u$ was smooth, we did not encounter an example where it was helpful to take $\eps$ equal to $\tau$ plus 
a `mesh-dependent' term that approximates $\min_{z \in X^\delta}\| u-z\|_X$. A probable explanation is the better conditioning of the discretized problems on coarser meshes.
\end{remark}

\begin{remark}[The case that $(A,L)$ is only closed]
This section started by assuming a pair $(A,L) \in \cL(X,V \times H)$. If for some Hilbert space $\tilde X$, $(A,L)\colon \tilde X \supset \dom(A,L)\rightarrow V \times H$ is only linear and closed, then by defining $X:=\{z\in \tilde{X}\colon (Az,Lz) \in V \times H\}$ equipped with the graph norm, we are back in the situation required for Assumption~\ref{assump}.
\end{remark}

\begin{remark}[About closedness of $\ran(A,L)$] \label{rem:closedrange} By injectivity of $(A,L)$ from Assumption~\ref{assump}, the open mapping theorem shows that for any $\eps>0$, closedness of $\ran(A,\eps L)\subset V \times H$ is equivalent to $(A,\eps L) \in \Lis(X,\ran(A,\eps L))$, being equivalent to $\nrm\cdot\nrm_\eps \eqsim \|\cdot\|_X$ (obviously generally \emph{dependent} on $\eps$). The latter shows 
that closedness of $\ran(A,\eps L)$ implies that $(X,\nrm\cdot\nrm_\eps)$ is a Hilbert space, as well as that closedness of $\ran(A,\eps L)$ is equivalent to closedness of $\ran(A,L)$.

Conversely, if $\ran(A,L)$ is \emph{not} closed, so that $\nrm\cdot\nrm_\eps$ is \emph{not} equivalent $\|\cdot\|_X$, then from $\nrm\cdot\nrm_\eps \lesssim \|\cdot\|_X$ and the open mapping theorem it follows that $(X,\nrm \cdot\nrm_\eps)$ is \emph{not} a Hilbert space.

An advantage offered by a pair $(A,L)$ with closed $\ran(A,L)$ is that one can bound the condition number of the linear system that determines this approximation (see Remark~\ref{conditioning}). In that context note that
if Assumption~\ref{assump} holds for some $(A,L) \in  \cL(X,V \times H)$, then it holds
also for $(A,\identity) \in  \cL(X,V \times X)$ with $j^{\rm new}:=j$ and $\eta_{\mathcal{C}}^{\rm new}:=\eta_{{\mathcal{C}}\|L\|_{\cL(V,H)}}$, where now $\ran (A,\identity)$ \emph{is} closed. Despite this advantage we will not insist on closedness of $\ran(A,L)$ in what follows. One argument for that is the following. If unconditional stability holds, then for computing $u_\eps^\delta$ using $(A,0)$ 
no regularization will be needed. 
\end{remark}

\section{Applications} \label{Sappls}
\begin{example}[Cauchy problem for Poisson's equation] \label{ex1} 
Let $\Omega \subset \R^d$ be a Lipschitz domain, and let $\Sigma$ and $\Sigma^c$ be open, measurable subsets of $\partial\Omega$ with $\Sigma\cap \Sigma^c=\emptyset$, $\overline{\Sigma} \cup \overline{\Sigma^c}=\partial\Omega$, and $|\Sigma|>0$.
\corr{Informally,}  the Cauchy problem asks for finding a solution to
\begin{equation}
\label{informal}
-\triangle u = f_I \text{ on } \Omega,\quad u = f_D  \text{ on } \Sigma ,\quad \tfrac{\partial u}{\partial n} = f_N \text{ on } \Sigma,
\end{equation}
i.e., Dirichlet and Neumann conditions are imposed on the same boundary portion of non-vanishing measure. \corr{To avoid unnecessarily restrictive assumptions on the data under which the formulation \eqref{informal} would be meaningful, and}  to identify an appropriate least squares functional
we  employ the following \corr{rigorous} weak
 formulation.
Given $f=(f_I,f_D,f_N) \in H^1_{0,\Sigma^c}(\Omega)' \times H^{\frac12}(\Sigma) \times H^{-\frac12}(\Sigma)$, 
 with $g_{f_I,f_N}:=v \mapsto f_I(v)+\int_{\Sigma} f_N v\,ds \in H^1_{0,\Sigma^c}(\Omega)'$, \corr{we search a solution of}
\be 
\label{Cauchy}
A u = (B_1 u,B_2u)=(g_{f_I,f_N},f_D),
\ee
where
$(B_1, B_2) \in \cL\big( \underbrace{H^1(\Omega)}_{X:=},  \underbrace{H^1_{0,\Sigma^c}(\Omega)' \times H^{\frac12}(\Sigma)}_{V:=}\big)$
is defined by $B_2:=\gamma_\Sigma$, being the trace operator on $\Sigma$, and
$B_1$ represents the negative Laplacian as a mapping from $X$ to $H^1_{0,\Sigma^c}(\Omega))'$
$$
(B_1z)(v):=\int_{\Omega} \nabla z \cdot \nabla v \,dx \qquad (z \in X,\,v \in H^1_{0,\Sigma^c}(\Omega)).
$$
Here $H^1_{0,\Sigma^c}(\Omega)$ is the closure in $H^1(\Omega)$ of the smooth functions on $\Omega \cup \Sigma$ with compact support, and
$H^{-\frac12}(\Sigma)$ is the dual of $\widetilde{H}^{\frac12}(\Sigma)$, the latter in the literature also denoted by $H_{00}^{\frac12}(\Sigma)$.
The dual of $H^{\frac12}(\Sigma)$ is denoted by $\widetilde{H}^{-\frac12}(\Sigma)$.

For this problem it is known that $\ker A=0$ and, when $|\Sigma^c|>0$, $\ran A \subsetneq \overline{\ran A}=V$ (e.g.~\cite[Prop.~3.2]{28.5}).
In \cite[Thm.~1.7, Rem.~1.8, and Thm.~1.9]{10.1}, the following conditional \emph{interior} and conditional \emph{global} stability results have been established:
{\renewcommand{\theenumi}{\roman{enumi}}
\begin{enumerate}
\item \label{i} 
For $\omega \subset \Omega$ with $\dist(\omega,\Sigma^c)>0$, there exists a $\sigma \in (0,1)$ such that
$$
\|z\|_{L_2(\omega)} \lesssim (\|A z\|_V+\|z\|_{L_2(\Omega)})^{1-\sigma} \|A z\|_V^{\sigma} \qquad(z \in X).
$$
\item \label{ii} There exists a $\sigma \in (0,1)$ such that for $z\neq 0$
$$
\|z\|_{L_2(\Omega)} \lesssim  (\|A z\|_V+\|z\|_{H^1(\Omega)}) \Big(\log\big(1+\frac{\|z\|_{H^1(\Omega)}}{\|A z\|_V}\big)\Big)^{-\sigma} \qquad(z \in X).
$$
\end{enumerate}}
\noindent Notice that \eqref{i} and \eqref{ii} are of the form as in Assumption~\ref{assump} where $\|Lz\|_H$ and $\eta_{\mathcal{C}}(t)$
read as $\|z\|_{L_2(\Omega)}$ and $\mathcal{O}\big((t+{\mathcal{C}})^{1-\sigma}t^\sigma\big)$, or $\|z\|_{H^1(\Omega)}$ and $\mathcal{O}\big((t+\mathcal{C})\big(\log (1+\mathcal{C}/t)\big)^{-\sigma} \big)$, respectively.
\end{example}

\begin{example}[Data-assimilation for the heat equation] \label{ex2}
Let $\Omega \subset \R^d$ be a domain, $0<T_1<T_2< T$, and $\emptyset \neq \omega \subset \Omega$ open. With $I:=(0,T)$,
given $(f,g) \in L_2(I;H^{-1}(\Omega)) \times L_2(I \times \omega)$, the data-assimilation problem reads as finding  $u$ with
$\partial_t u -\triangle_x u=f$ on $I \times \Omega$, and $u|_{I \times \omega}=g$, or, more precisely,
$Au=(f,g)$, where $A=(B,\Gamma_{I \times \omega}) \in \cL\big(\underbrace{L_2(I;H^1(\Omega)) \cap H^1(I;H^{-1}(\Omega))}_{X:=},\underbrace{L_2(I;H^{-1}(\Omega)) \times L_2(I \times \omega)}_{V:=}\big)$ is defined by $\Gamma_{I \times \omega} z=z|_{I \times \omega}$, and 
$(Bz)(v):=\int_I \int_\Omega \partial_t z \,v +\nabla_x z \cdot \nabla_x v\,dx\,dt$.

The following conditional stability estimates can be found in \cite{35.925}:
{\renewcommand{\theenumi}{\alph{enumi}}
\begin{enumerate}
\item \label{hi} For a bounded $\breve{\omega} \Subset \Omega$, there exists a $\sigma \in (0,1)$ such that  
$$
\|z\|_{L_2((T_1,T_2);H^1(\breve{\omega}))} \lesssim 
\big(\|A z\|_{V}+\|z\|_{L_2(I \times \Omega)}\big)^{1-\sigma} \|A z\|_{V}^\sigma \qquad (z \in X).
$$
\item \label{hii} When one has the additional information that $u=0$ on $I \times \partial \Omega$, then $X$ should be redefined as $X:=L_2(I;H_0^1(\Omega)) \cap H^1(I;H^{-1}(\Omega))$. For $\Omega$ being a bounded convex polytope, now 
it holds that
$$
\|z\|_{L_2((T_1,T);H^1(\Omega)) \cap H^1((T_1,T);H^{-1}(\Omega)) } \lesssim \|A z\|_{V} \qquad (z \in X).
$$
\end{enumerate}}
\noindent  In Case \eqref{hi}, Assumption~\ref{assump} is valid with
$\|Lz\|_H$ and $\eta_{\mathcal{C}}(t)$ reading as $\|z\|_{L_2(I \times \Omega)}$ and $\mathcal{O}\big((t+\mathcal{C})^{1-\sigma}t^\sigma\big)$, whilst Case \eqref{hii} concerns \emph{unconditional} (Lipschitz) stability, i.e., $L=0$ and $\eta(t)=\mathcal{O}(t)$.
To show the latter, it remains to verify that $A$ is injective in Case \eqref{hii}.
Suppose it is not, and let $0 \neq z\in X$ with $\|Az\|_V=0$. From $X \hookrightarrow C([0,T];L_2(\Omega))$ (\cite[Ch.~1, Thm.~ 3.1]{185}), there exists an open interval $J \subset (0,T)$ such that for any $t \in J$, $\|z(t,\cdot)\|_{L_2(\Omega)} \neq 0$. 
On the other hand, the above estimate shows that  $\|z\|_{L_2((T_1,T);H^1_0(\Omega)) \cap H^1((T_1,T);H^{-1}(\Omega)) }=0$ for any $T_1 \in (0,T)$.
From $L_2((T_1,T);H^1_0(\Omega)) \cap H^1((T_1,T);H^{-1}(\Omega)) \hookrightarrow C([T_1,T];L_2(\Omega))$ we arrive at a contradiction.
\end{example}

\begin{example}[Data-assimilation for the wave equation, {\cite[Remark~A.5]{35.9296}}] 
\label{ex3}
\mbox{}Let $\Omega \subset \R^d$ be a domain with a smooth boundary, and, for some $T>0$, let $I:=(0,T)$.
Let $\omega \subset \overline{\Omega}$, and assume that $I \times \omega$ satisfies the \emph{Geometric Control Condition} \cite{19.3,19.4}.
Roughly speaking, it means that all geometric optic rays in $I \times \Omega$, taking into account their reflections at the lateral boundary, intersect the set $I \times \omega$.

Given $(f,g,h) \in V:= H^{-1}(I \times \Omega) \times L_2(I\times \partial\Omega) \times L_2(I\times \omega)$, the data assimilation problem reads as finding $u$
that satisfies
\begin{align*}
&(\Box u)(v):=\int_I \int_\Omega -\partial_t u \,\partial_t v+\nabla_x u \cdot \nabla_x v \,dx\,dt=f(v)\quad (v \in H_0^1(I\times \Omega)),\\
& \gamma_{I\times \partial\Omega}u:=u|_{I\times \partial\Omega}=g, \text{ and } \Gamma_{I\times \omega} u:= u|_{I\times \omega}=h.
\end{align*}
With $A:=(\Box, \gamma_{I\times \partial\Omega}, \Gamma_{I\times \omega})$
and $X=\{z \in L_2(I \times \Omega)\colon Az \in V\}$ equipped with the graph norm, or its completion when $A\colon L_2(I \times \Omega)\supset D_A \rightarrow V$ is not closed, 
the following \emph{unconditional (Lipschitz) stability} is valid:
\be \label{22}
\|z\|_{L_\infty(I;L_2(\Omega))}+\|\partial_t z\|_{L_2(I;H^{-1}(\Omega))} \lesssim \|A z\|_V  \qquad (z \in X). \qedhere
\ee
\end{example}

\begin{remark}[The condition of $\partial\Omega$ being smooth] \label{rem:polytope} For finite element computations, for $d>1$ a problem with the setting of Example~\ref{ex3} is the condition of $\partial\Omega$ being smooth.
It is used in the derivation of both the \emph{Distributed Observability Estimate} $\|z(0,\cdot)\|_{L_2(\Omega)}+\|\partial_t z(0,\cdot)\|_{L_2(\Omega)} \lesssim \|\Gamma_{I\times \omega} z\|_{L_2(I\times \omega)}$ for functions $z$ that satisfy $\Box z=0$ and $\gamma_{I\times \partial\Omega} z=0$, see e.g.~\cite{169.0555}, and the 
\emph{Energy Estimate}
\begin{align*}
\|z\|&_{L_\infty(I;L_2(\Omega))}+\|\partial_t z\|_{L_2(I;H^{-1}(\Omega))}\lesssim\\
&\|z(0,\cdot)\|_{L_2(\Omega)}+\|\partial_t z(0,\cdot)\|_{H^{-1}(\Omega)}+\|\gamma_{I\times \partial\Omega} z\|_{L_2(I \times \partial\Omega)}+\|\Box z\|_{H^{-1}(I \times \Omega)},
\end{align*}
see \cite[Prop.~A.1]{35.9296}, which builds on \cite[Thm.~2.1 and 2.3]{169.066}.
From both these estimates, one easily derives \eqref{22} (cf.~\cite[proof of Thm.~2.2]{35.859} (see however \cite[Rem.~2.6]{35.9296})).
In \cite{35.859} it is claimed that under stronger conditions on $\omega$, the Distributed Observability Estimate can also be valid for polytopal $\Omega$.
See also \cite{35.8598} for the related boundary controllability of the wave equation on domains with corners.
Assuming that on such domains also an energy estimate is valid, possibly with weaker norms
 on the left-hand side, in the case that these norms do control $\|z\|_{L_2(I \times \Omega))}$, \cite[proof of Thm.~2.2]{35.859} will give an unconditional stability estimate \eqref{22} with those norms on the left-hand side.
\end{remark}

Notice that only in Example~\ref{ex1}\eqref{ii}, $\ran(A,L)$ is closed.

Other examples of conditional stability include, e.g., \emph{data-assimilation for the Poisson equation} (\cite{35.928}), \emph{the backward heat equation} (\cite{145}), and \emph{the heat and wave equations with lateral Cauchy data} (\cite{168.865} and \cite[Thm.~3.4.1]{168.825}).

\section{Least squares approximation} \label{LS}
We adhere to the assumptions in Section \ref{sec:2}.
For a given suitable finite dimensional space $X^\delta$, we
propose, as an approximate ``solution'' to \eqref{problem},  the unique minimizer $u_\eps^\delta\in X^\delta$
of the  \emph{regularized} least-squares functional 
$$z \mapsto\sqrt{\|A z-f\|_V^2+\eps^2\|Lz\|_H^2}.$$ As we have seen in applications, however, the space $V$, or one or more components of $V$ when it is a product space, are equipped with a norm that cannot be evaluated. For example it can be either a dual norm, i.e., a norm of type $\sup_{0 \neq v \in Y} \frac{|\cdot(v)|}{\|v\|_Y}$, or a fractional Sobolev norm, the latter which typically arises with the enforcement of boundary conditions.

For the dual norm case we will see that under an inf-sup or Ladyshenskaja-Babu\u{s}ka-Brezzi (LBB) condition, 
for minimizing the least-squares functional over $X^\delta$ the supremum over $Y$ can be replaced by a supremum over a  suitable finite dimensional space, which makes it computable assuming $\|\cdot\|_Y$ can be evaluated.

Dealing with reflexive spaces, any norm can be viewed as a dual norm. Applying this to a fractional Sobolev norm, the space $Y$ is a fractional Sobolev space with smoothness index of opposite sign. At a first glance this might not seem to be helpful, but we saw that under an LBB condition  the norm on the latter space needs to be evaluated for arguments from a finite dimensional subspace only. Moreover, as we will see, it suffices to compute a norm on this finite dimensional subspace that is only (uniformly) \emph{equivalent} to the norm on $Y$.

In view of above comments, we consider our least squares problem in the following setting that covers all envisioned scenarios. 
Specifically, $V$ may be the product of a dual space and a space whose norm is easy to evaluate. So, for some Hilbert spaces $Y$ and $W$, let $V=Y'\times W$, $A=(B,C)$, so that \eqref{trip} takes the form
\be \label{17}
\nrm\cdot\nrm_\eps:=\sqrt{\|B \cdot\|_{Y'}^2+\|C \cdot\|_{W}^2+\eps^2\|L\cdot\|_H^2}\,,
\ee
and let $f=(g,h)$, so that the aforementioned least-squares functional reads as $z \mapsto\sqrt{\|B z-g\|_{Y'}^2+\|C z-h\|_W^2+\eps^2\|Lz\|_H^2}$.
We assume that the $\|\cdot\|_W$-norm can be evaluated, ignoring possible quadrature issues. The analysis of cases where 
the triple $(V,A,f)$ has none or multiple components of type either  $(Y',B,g)$ or $(W,C,h)$ causes no additional problems.

To avoid the exact evaluation of the $\|\cdot\|_{Y'}$-norm, given a family $(X^\delta)_{\delta \in \Delta}$ of finite dimensional subspaces of $X$, let $(Y^\delta)_{\delta \in \Delta}$ be a family of finite dimensional subspaces of $Y$ such that
\be \label{inf-sup}
\varrho:=\inf_{\delta \in \Delta} \inf_{ {\{z \in X^\delta\colon B z \neq 0\}}}  \sup_{0 \neq v \in Y^\delta} \frac{|(B z)(v)|}{\|B z\|_{Y'} \|v\|_Y}>0,
\ee
and let $\langle \cdot,\cdot \rangle_{Y^\delta}$ be an inner product on $Y^\delta$ with associated norm $\|\cdot\|_{Y^\delta}$ that satisfies
\be \label{equiv}
\|\cdot\|_{Y^\delta} \eqsim \|\cdot\|_Y\quad \text{on } Y^\delta
\ee
(uniformly in $\delta \in \Delta$).
A choice $\|\cdot\|_{Y^\delta} \neq \|\cdot\|_Y$ is useful when also the $\|\cdot\|_Y$-norm cannot be (easily) evaluated. As we will see in Sect.~\ref{Sposdef}, 
another reason for taking a suitable $\|\cdot\|_{Y^\delta} \neq \|\cdot\|_Y$ is when the stiffness matrix corresponding to $\langle\cdot,\cdot\rangle_{Y^\delta}$ can be efficiently inverted  in which case the approximation $u_\eps^\delta$ defined in the next theorem can be found as the solution of a symmetric positive definite system instead of a saddle-point system.

\begin{theorem} \label{thm:1} Let \eqref{inf-sup}-\eqref{equiv} be valid. For $u_\eps^\delta$ being the unique solution of
\be \label{13}
u_\eps^\delta:=\argmin_{z \in X^\delta} \Big\{\sup_{0 \neq v \in Y^\delta} \frac{|(Bz-g)(v)|^2}{\|v\|_{Y^\delta}^2}
+\|Cz-h\|_{W}^2+\eps^2\|L z\|_H^2\Big\},
\ee
it holds that
\be \label{4}
\nrm u -u_\eps^\delta \nrm_\eps \lesssim \|Bu-g\|_{Y'}+\|Cu-h\|_{W}+\min_{z \in X^\delta}\nrm u-z\nrm_\eps+\eps \|Lu\|_H,
\ee
i.e., \eqref{1} is valid.
\end{theorem}

\begin{proof} 
Initially we assume that $\|\cdot\|_{Y^\delta}= \|\cdot\|_Y$
on $Y^\delta$, and postpone the discussion about the case where
we have only a uniform equivalence \eqref{equiv}
until the end of the proof.

The basic idea behind the proof is to use a suitable intermediary $u_\eps\in X$
to estimate then
$\nrm u-u_\eps\nrm_\eps$ and $\nrm u_\eps -u_\eps^\delta \nrm_\eps$. Ideally,
$u_\eps$ should be the minimizer over $z \in X$ of $\|Bz-g\|^2_{Y'}+\|Cz-h\|^2_{W}+\eps^2\|L z\|^2_H$. Since we are not insisting of $\ran(B,C,L)$ being closed
(cf. Remark~\ref{rem:closedrange}), this
minimizer would not necessarily exist. Therefore, more work is required to
identify such uniform intermediaries via a perturbation that
ensures closedness.
To that end, for $\zeta \geq 0$ we equip $X$ with the additional norm
$$
\nrm\cdot\nrm_{\eps,\zeta}:=\sqrt{\|B \cdot\|_{Y'}^2+\|C \cdot\|_{W}^2+\eps^2\|L\cdot\|_H^2+\zeta^2\|\cdot\|_X^2}.
$$
Since  $\ran(B,C,L,\identity)$ \emph{is} closed, for $\zeta>0$ 
the minimizer $u_{\eps,\zeta}$ over $z \in X$ of
$$
G_{\eps,\zeta}(z):= \|Bz-g\|_{{Y}'}^2+\|Cz-h\|_{W}^2+\eps^2\|L z\|_H^2+\zeta^2\|z\|_X^2
$$
does exist uniquely as the solution of the Euler-Lagrange equations
$$
\langle B \tilde{z},B u_{\eps,\zeta}-g\rangle_{Y'}+\langle C \tilde{z},Cu_{\eps,\zeta}-h\rangle_W+\eps^2 \langle L \tilde{z},L u_{\eps,\zeta}\rangle_H+
\zeta^2 \langle \tilde{z},u_{\eps,\zeta}\rangle_X=0 \quad(\tilde{z} \in X).
$$ 

To deal with the inner product in $Y'$ we employ the Riesz \emph{lifter} ($=$ inverse Riesz \emph{map}) $R_{Y}\colon Y' \rightarrow Y$, defined for $f\in Y'$ by 
$$
\langle \tilde{z},R_{Y}f\rangle_Y=
f(\tilde{z}) ,\quad \forall\, \tilde{z} \in Y,
$$ 
and introduce the  lifted residual
\be
\label{RL}
v_{\eps,\zeta}:=R_{Y}(g-B u_{\eps,\zeta}).
\ee
We can then write
$$
\langle B \tilde{z},B u_{\eps,\zeta}-g\rangle_{Y'}=- \langle v_{\eps,\zeta}, R_{Y} B \tilde{z}\rangle_Y=-(B \tilde{z})(v_{\eps,\zeta}),
$$
from which one infers that $(u_{\eps,\zeta},v_{\eps,\zeta}) \in X \times Y$ is the (unique) solution of 
\begin{align*}
&\underbrace{(B \tilde{z})(v_{\eps,\zeta})
+(Bu_{\eps,\zeta})(\tilde{v})+\langle \tilde{v},v_{\eps,\zeta}\rangle_Y
-\langle C \tilde{z},C u_{\eps,\zeta}\rangle_{W}-\eps^2\langle L \tilde{z},L u_{\eps,\zeta}\rangle_H-\zeta^2\langle \tilde{z},u_{\eps,\zeta}\rangle_X}_{d_{\eps,\zeta}((u_{\eps,\zeta},v_{\eps,\zeta}),(\tilde{z},\tilde{v})):=\hspace*{20em}}\\
&\hspace*{15em}
=\ell(\tilde{z},\tilde{v}):=g(\tilde{v})-\langle C \tilde{z},h\rangle_{W} \quad((\tilde{z},\tilde{v}) \in X \times Y).
\end{align*}
The symmetric bilinear form $d_{\eps,\zeta}$ on $((X,\nrm\cdot\nrm_{\eps,\zeta})\times Y)\times ((X,\nrm\cdot\nrm_{\eps,\zeta})\times Y)$ is bounded (uniformly in $\eps>0$ and $\zeta \geq 0$).

Similarly as in the continuous case, for $\delta \in \Delta$ the minimizer $u^\delta_{\eps,\zeta}$ over $X^\delta$ of 
$$
G^\delta_{\eps,\zeta}(z):=\sup_{0 \neq v \in Y^\delta} \frac{|(Bz-g)(v)|^2}{\|v\|_{Y}^2}
+\|Cz-h\|_{W}^2+\eps^2\|L z\|_H^2+\zeta^2\|z\|_X^2
$$
exists uniquely.
Thanks to $\dim X^\delta<\infty$, this holds even true for $\zeta=0$ (where $u^\delta_{\eps,0}=u^\delta_\eps$ is the solution of \eqref{13})

Defining in analogy to \eqref{RL},
$$
\langle \tilde z,v^\delta_{\eps,\zeta}\rangle_Y= (g-B u^\delta_{\eps,\zeta})(\tilde z),
\quad \tilde z\in Y^\delta,
$$
the pair  $(u^\delta_{\eps,\zeta},v^\delta_{\eps,\zeta}) \in X^\delta \times Y^\delta$ solves the Galerkin system
\be 
\label{5}
d_{\eps,\zeta}\big((u^\delta_{\eps,\zeta},v^\delta_{\eps,\zeta}),(\tilde{z},\tilde{v})\big)=\ell(\tilde{z},\tilde{v}) \quad((\tilde{z},\tilde{v}) \in X^\delta \times Y^\delta).
\ee
We will demonstrate later below that
\be \label{6}
\inf_{\delta \in \Delta,\,\eps>0,\,\zeta\geq 0} \inf_{0 \neq (z,v)\in X^\delta \times Y^\delta} \sup_{0 \neq (\tilde{z},\tilde{v})\in X^\delta \times Y^\delta} \frac{d_{\eps,\zeta}\big((z,v),(\tilde{z},\tilde{v})\big)}{\sqrt{\nrm z\nrm_{\eps,\zeta}^2+\|v\|_Y^2}\sqrt{\nrm \tilde{z}\nrm_{\eps,\zeta}^2+\|\tilde{v}\|_Y^2}}>0.
\ee
Assuming the validity of \eqref{6} for the moment,  the symmetry and boundedness of $d_{\eps,\zeta}$, then shows 
that for $\zeta>0$,
\be \label{11}
\|v_{\eps,\zeta} -v^\delta_{\eps,\zeta}\|_Y+ \nrm u_{\eps,\zeta} -u^\delta_{\eps,\zeta}\nrm_{\eps,\zeta}
\lesssim \min_{(\tilde{z},\tilde{v}) \in X^\delta \times Y^\delta}\big\{
 \|v_{\eps,\zeta} -\tilde{v}\|_Y+\nrm u_{\eps,\zeta} -\tilde{z}\nrm_{\eps,\zeta}\big\}.
 \ee
uniformly in $\eps$ (see e.g., \cite[Thm.~3.1]{249.99}).
 
Now, for any $u\in X$ we simply estimate $\nrm u -u^\delta_{\eps,\zeta}\nrm_{\eps,\zeta} \leq \nrm u -u_{\eps,\zeta}\nrm_{\eps,\zeta}+\nrm u_{\eps,\zeta} -u^\delta_{\eps,\zeta}\nrm_{\eps,\zeta}$ and deduce first from \eqref{11} that
we have for the second term
\begin{align} 
\label{2term}
\nrm u_{\eps,\zeta} -u^\delta_{\eps,\zeta}\nrm_{\eps,\zeta}
&\lesssim  \|v_{\eps,\zeta}\|_Y+  \min_{\tilde{z} \in X^\delta} \nrm u_{\eps,\zeta} -\tilde{z}\nrm_{\eps,\zeta}\nonumber\\ 
& \leq \|g-Bu_{\eps,\zeta}\|_{Y'}+  \nrm u-u_{\eps,\zeta}\nrm_{\eps,\zeta}+  \min_{\tilde{z} \in X^\delta} \nrm u -\tilde{z}\nrm_{\eps,\zeta}\nonumber\\ 
&  \leq \|g-Bu\|_{Y'}+ 2 \nrm u-u_{\eps,\zeta}\nrm_{\eps,\zeta}+  \min_{\tilde{z} \in X^\delta} \nrm u -\tilde{z}\nrm_{\eps,\zeta}.
\end{align}
Regarding the first term, 
an application of the triangle inequality and optimality of $u_{\eps,\zeta}$
give
\begin{align*}
 \nrm u -u_{\eps,\zeta}\nrm_{\eps,\zeta} &\leq \sqrt{G_{\eps,\zeta}(u)}+\sqrt{G_{\eps,\zeta}(u_{\eps,\zeta})}\\
 & \leq 2 \sqrt{G_{\eps,\zeta}(u)}
 = 2 \sqrt{\|Bu-g\|_{Y'}^2+\|Cu-h\|_W^2+\eps^2 \|L u\|_H^2+\zeta^2 \|u\|_X^2},
\end{align*}
which together with \eqref{2term} yields
\be \label{23}
\nrm u -u_{\eps,\zeta}^\delta \nrm_{\eps,\zeta} \lesssim \|Bu-g\|_{Y'}\!+\!\|Cu-h\|_{W}\!+\!\min_{z \in X^\delta}\nrm u-z\nrm_{\eps,\zeta}\!+\!\eps \|Lu\|_H\!+\!\zeta\|u\|_X
\ee
(\emph{uniformly} in $\delta \in \Delta$, $\eps,\zeta>0$).

Now \emph{fixing} $\delta \in \Delta$ and $\eps>0$, the norms $\nrm\cdot\nrm_{\eps,\zeta}$ on $X^\delta$ are equivalent uniformly in $\zeta \in [0,1]$,
and $\lim_{\zeta \downarrow 0} \min_{z \in X^\delta}\nrm u-z\nrm_{\eps,\zeta}=\min_{z \in X^\delta}\nrm u-z\nrm_{\eps}$.
Introducing the operator  $D^\delta_{\eps,\zeta}\in \cL(X^\delta\times Y^\delta,(X^\delta\times Y^\delta)')$, defined by
$$
(D^\delta_{\eps,\zeta}(u^\delta_{\eps,\zeta},v^\delta_{\eps,\zeta}))(\tilde{z},\tilde{v}):=
d_{\eps,\zeta}\big((u^\delta_{\eps,\zeta},v^\delta_{\eps,\zeta}),(\tilde{z},\tilde{v})\big),
$$
and writing $(D^\delta_{\eps,\zeta})^{-1}=(D^\delta_{\eps,0})^{-1}+(D^\delta_{\eps,\zeta})^{-1}(D^\delta_{\eps,0}-D^\delta_{\eps,\zeta})(D^\delta_{\eps,0})^{-1}$, the uniform stability \eqref{6} and $\lim_{\zeta \downarrow 0} D^\delta_{\eps,\zeta}=D^\delta_{\eps,0}$ show  that $\lim_{\zeta \downarrow 0} u_{\eps,\zeta}^\delta=u_{\eps}^\delta$.
We conclude that by taking the limit for $\zeta \downarrow 0$ in \eqref{23} 
the proof of \eqref{4} for the case that $\|\cdot\|_{Y^\delta}=
\|\cdot\|_Y$ is completed as soon as we have confirmed the validity of \eqref{6}.

The latter statement \eqref{6} left to be shown is equivalent to uniform stability 
of the variational problem \eqref{5}, i.e.,
\be \label{7}
\nrm u_{\eps,\zeta}^\delta\nrm_{\eps,\zeta}+\|v_{\eps,\zeta}^\delta\|_Y \lesssim \|\ell\|_{(X^\delta,\nrm \cdot\nrm_{\eps,\zeta})\times Y^\delta)'}
\ee
(uniformly in $\delta$, $\eps>0$, and $\zeta \geq 0$) which requires 
utilizing \eqref{inf-sup}. To that end, we introduce as additional variables 
$\theta=-C u_{\eps,\zeta}^\delta$, $\mu=-\eps L u_{\eps,\zeta}^\delta$, and $\chi=-\zeta u_{\eps,\zeta}^\delta$ and the bilinear form
\be \label{eq:e}
e(z,(\tilde{v},\tilde{\theta},\tilde{\mu},\tilde{\chi})):=(B z)(\tilde{v})+\langle \tilde{\theta},C z\rangle_W+\eps \langle \tilde{\mu}, L z\rangle_H+\zeta\langle \tilde{\chi},z\rangle_X
\ee
over $X^\delta\times (Y^\delta\times W\times  H\times X)$.
Then, \eqref{5} is equivalent to
finding $(u^\delta_{\eps,\zeta},v^\delta_{\eps,\zeta},\theta,\mu,\chi) \in X^\delta \times Y^\delta \times W \times H \times X$ such that
\be \label{8}
\begin{split}
& \langle (\tilde{v},\tilde{\theta},\tilde{\mu},\tilde{\chi}),(v^\delta_{\eps,\zeta},\theta,\mu,\chi)\rangle_{Y \times W \times H \times X}+e(u_{\eps,\zeta}^\delta,(\tilde{v},\tilde{\theta},\tilde{\mu},\tilde{\chi}))\\
&\hspace*{6cm}  +
e(\tilde{z},(v_{\eps,\zeta}^\delta,\theta,\mu,\chi))=\ell(\tilde{z},\tilde{v})
\end{split}
\ee
for all $(\tilde{z},\tilde{v},\tilde{\theta},\tilde{\mu},\tilde{\chi}) \in X^\delta \times Y^\delta \times W \times H \times X$.  

Recall from \eqref{inf-sup} that for any $\sigma \in (0,1)$, given $z \in X^\delta$ we can find $\tilde{v} \in Y^\delta$ with $\|\tilde{v}\|_Y=\|Bz\|_{Y'}$ and $ (B z)(\tilde{v}) \geq \sigma \rho \|Bz\|_{Y'}^2$. Then
take $(\tilde{\theta},\tilde{\mu},\tilde{\chi})=(Cz,\eps L z,\zeta z)$, to
conclude that 
\begin{align*}
 e(z,(\tilde{v},\tilde{\theta},\tilde{\mu},\tilde{\chi}))
&\geq \sigma \rho \|Bz\|_{Y'}^2 +\|C z\|_W^2+\eps^2\|L z\|_H^2+\zeta^2\|z\|_X^2
\\
&\geq \sigma \varrho \nrm z\nrm_{\eps,\zeta}  \sqrt{
\|\tilde{v}\|_Y^2 +\|\tilde{\theta}\|_W^2+\|\tilde{\mu}\|_H^2+\|\tilde{\chi}\|_X^2}.
\end{align*}
From this LBB stability we conclude that \eqref{8} is indeed uniformly stable, i.e., that $\nrm u^\delta_{\eps,\zeta}\nrm_{\eps,\zeta}+\|v^\delta_{\eps,\zeta}\|_Y+\|\theta\|_W+\|\mu\|_H +\|\chi\|_X \lesssim  \|\ell\|_{(X^\delta,\nrm \cdot\nrm_{\eps,\zeta})\times Y^\delta)'}$, which implies \eqref{7}.

Finally, we discuss the case that $\|\cdot\|_{Y^\delta}\neq \|\cdot\|_Y$. We write $Y=Y^\delta \oplus (Y^\delta)^{\perp_Y} \simeq Y^\delta \times (Y^\delta)^{\perp_Y}$, and replace for the $Y^\delta$-component the inner product $\langle\cdot,\cdot\rangle_Y$ by $\langle \cdot,\cdot\rangle_{Y^\delta}$.
Then we are back in the case that $\|\cdot\|_{Y^\delta}= \|\cdot\|_Y$ on $Y^\delta$, and the proof so far shows that $u_\eps^\delta$ defined in \eqref{13} satisfies \eqref{4} with the $Y'$-norm at its right-hand side and in the definition of $\nrm\cdot\nrm_\eps$ at its left- and right-hand side being defined in terms of the so modified $Y$-norm.
Since by our assumption \eqref{equiv} this modified $Y$-norm, and so its resulting dual norm, is equivalent to the original norm, uniformly in $\delta \in \Delta$, the proof is completed.
\end{proof}

In view of \eqref{5} and the last paragraph in the above proof, note that the least squares approximation $u_\eps^\delta$ defined in \eqref{13}  can be computed as the first component of the pair $(u_\eps^\delta,v_\eps^\delta) \in X^\delta \times Y^\delta$ that solves the \emph{mixed system}
\be \label{14}
\framebox{$\begin{aligned}
&(B \tilde{z})(v_\eps^\delta)
+(B u_\eps^\delta)(\tilde{v})+\langle \tilde{v},v_\eps^\delta\rangle_{Y^\delta}\\
&-\langle C \tilde{z},C u_\eps^\delta\rangle_{W}-\eps^2\langle L \tilde{z},L u_\eps^\delta\rangle_H=g(\tilde{v})-\langle C \tilde{z},h\rangle_{W} \quad((\tilde{z},\tilde{v}) \in X^\delta \times Y^\delta).
\end{aligned}$}
\ee
For a suitable choice of $\langle \cdot,\cdot\rangle_{Y^\delta}$, in Sect.~\ref{Sposdef} we will reduce this system to a symmetric positive definite system for the primal variable $u_\eps^\delta$.

\begin{remark}[Relation to existing work] 
Suppose that the norm $\|\cdot\|_X$ is computable (which is not the case
in Example~\ref{ex2}) and assume for simplicity that $\|\cdot\|_{Y^\delta}=\|\cdot\|_Y$. Then by replacing  $\eps^2\|L z\|_H^2$ by the stronger regularizing term $\eps^2\|z\|_X^2$,
instead of minimizing \eqref{13} as in Theorem~\ref{thm:1} one may minimize the least-squares functional
$$
\hat{u}_\eps^\delta:=\argmin_{z \in X^\delta} \Big\{\sup_{0 \neq v \in Y^\delta} \frac{|(Bz-g)(v)|^2}{\|v\|_{Y}^2}
+\|Cz-h\|_{W}^2+\eps^2\|z\|_X^2\Big\}.
$$
Upon  replacing $\eps^2 \langle L \cdot,L\cdot\rangle_H$ by $\eps^2 \langle \cdot,\cdot\rangle_X$  in \eqref{5}, and the test function $\tilde{z}$ by $-\tilde{z}$,
this amounts to
  solving for $(\hat{u}_\eps^\delta,\hat{v}_\eps^\delta) \in X^\delta \times Y^\delta$ from
\be
\begin{split}
 \label{10}
-(B \tilde{z})&(\hat{v}_\eps^\delta)
+(B \hat{u}_\eps^\delta)(\tilde{v})+\langle \tilde{v},\hat{v}_\eps^\delta\rangle_{Y}
\\&+\langle C \tilde{z},C \hat{u}_\eps^\delta\rangle_{W}
+\eps^2\langle \tilde{z},\hat{u}_\eps^\delta\rangle_X=g(\tilde{v})+\langle C \tilde{z},h\rangle_{W} \quad((\tilde{z},\tilde{v}) \in X^\delta \times Y^\delta).
\end{split}
\ee
The bilinear form on $X \times Y$ on the left-hand side of \eqref{10} is bounded, uniformly in $\delta \in \Delta$.
Moreover, already \emph{without} assuming \eqref{inf-sup}, it is also \emph{coercive}, although with the unfavorable coercivity
 constant $\eqsim \eps^2$.
With $(\hat{u}_\eps,\hat{v}_\eps) \in X \times Y$  denoting the solution of \eqref{10} under testing with all $(\tilde{z},\tilde{v}) \in X \times Y$, one arrives at the estimate
\be \label{12}
\|\hat{v}_\eps -\hat{v}^\delta_\eps\|_Y+ \| \hat{u}_\eps -\hat{u}^\delta_\eps\|_X \lesssim \eps^{-2} \inf_{(\tilde{z},\tilde{v}) \in X^\delta \times Y^\delta}
\{ \|\hat{v}_\eps -\tilde{v}\|_Y+\| \hat{u}_\eps -\tilde{z}\|_X\}.
 \ee

This approach to solve \eqref{10} without ensuring the LBB stability \eqref{inf-sup} has been suggested in \cite{19.897,28.5}. 
In \cite[Thm.~2.7]{28.5} it was shown that $\lim_{\eps \downarrow 0} (\hat{u}_\eps,\hat{v}_\eps) = (u,0)$ in $X \times Y$, where $u$ is the minimal norm solution of $(Bu,Cu)=(g,h)$, under the assumption that a solution exists. 
To conclude from \eqref{12} convergence of numerical approximations $u_\eps^\delta$ towards $u$ requires making assumptions on 
 the behavior of higher order norms of $(\hat{u}_\eps,\hat{v}_\eps)$ as functions of $\eps \downarrow 0$ (cf. \cite[Thm.~4.1]{19.897}). On the other hand, this approach does not assume any conditional stability, and it concerns convergence in norm instead of convergence in the functional $j$ associated to a conditional stability assumption.

Notice, however, that if conditional stability is available and the weaker notion of convergence in the associated functional is relevant,
then the difference between \eqref{12} and the quasi-optimal bound \eqref{11} shows that is indeed very rewarding
 to select $Y^\delta$ dependent on $X^\delta$ such that \eqref{inf-sup} is valid. 
 \end{remark}

Uniform discrete inf-sup stability \eqref{inf-sup} plays a  pivotal role in our approach. We recall that it is equivalent to the existence of uniformly bounded `Fortin operators'. Precisely, the following statement holds.

\begin{theorem}[{\cite[Prop.~5.1]{249.992}}] \label{thm:Fortin} For general $B \in \cL(X,Y')$, and closed subspaces $X^\delta$ and $Y^\delta$ of Hilbert spaces $X$ and $Y$ with $B X^\delta \neq \{0\}$ and $Y^\delta \neq \{0\}$, let
\be \label{fortin}
Q^\delta \in \cL(Y,Y^\delta) \text{ with } (B X^\delta)\big((\identity - Q^\delta)Y\big)=0.
\ee
Then $\xi^\delta := \inf_{\{z \in X^\delta\colon B z \neq 0\}} \sup_{0 \neq v \in Y^\delta} \frac{|(B z)(v)|}{\|B z\|_{Y'} \|v\|_Y} \geq \|Q^\delta\|_{\cL(Y,Y)}^{-1}$.

Conversely, when $\xi^\delta>0$, and $\ran B$ is closed or $\dim X^\delta <\infty$, then there exists a $Q^\delta$ as in \eqref{fortin}, being even a projector, with
$\|Q^\delta\|_{\cL(Y,Y)} = 1/\xi^\delta$.
\end{theorem}

Using this theorem, in Section~\ref{sec:inf-sup} we will verify condition \eqref{inf-sup} for the examples from Section \ref{Sappls}. 

Finally, we note that the condition \eqref{inf-sup} may be slightly relaxed:

\begin{remark}[Slight relaxation of inf-sup condition \eqref{inf-sup}] 
\label{rem1}
By an application of Young's inequality, it is not difficult to show that the reasoning that provided the LBB stability of the bilinear form $e$ from \eqref{eq:e}, and 
consequently the proof of Theorem~\ref{thm:1}, still applies when, for some constant $\mu <2\sqrt{\varrho}$, \eqref{inf-sup} is relaxed to $ \sup_{0 \neq v \in Y^\delta} \frac{|(B z)(v)|}{ \|v\|_Y} \geq \varrho \|B z\|_{Y'}-\mu\sqrt{\|Cz\|_W^2+\eps^2\|Lz\|_H^2}$.
\end{remark}

\section{Choice of $\langle \cdot,\cdot\rangle_{Y^\delta}$, and the reformulation of  \eqref{14} as a symmetric positive definite system} \label{Sposdef}
Let $G_{Y}^\delta ={G_{Y}^\delta}' \in \Lis({Y^\delta}',Y^\delta)$ be such that 
\be \label{Gd}
\|G_{Y}^\delta f \|_Y^2 \eqsim f(G_{Y}^\delta f) \quad(f \in {Y^\delta}',\,\delta \in \Delta).
\ee
Assuming its application can be computed `efficiently',  such an operator $G_{Y}^\delta$ is often called a (uniform) \emph{preconditioner}. See Remark \ref{rem:available} below for scenarios where such preconditioners are
available. This
$G_{Y}^\delta$ can be employed to \emph{define}
\be \label{16}
\langle v, \tilde{v}\rangle_{Y^\delta}:=((G_{Y}^\delta)^{-1} \tilde{v})(v)\quad (v,\tilde{v} \in Y^\delta),
\ee
which, in view of \eqref{Gd}, yields $\|\cdot\|_{Y^\delta} \eqsim \|\cdot\|_Y$ on $Y^\delta$, i.e., \eqref{equiv}.
The definition of $\langle \cdot,\cdot\rangle_{Y^\delta}$ shows that $G_{Y}^\delta$ is the Riesz lifter
associated to the Hilbert space $(Y^\delta,\langle \cdot,\cdot\rangle_{Y^\delta})$. Since a Riesz lifter is an isometry, it follows that
$$
\sup_{0 \neq v \in Y^\delta} \frac{|f(v)|^2}{\|v\|_{Y^\delta}^2} = \|G_{Y}^\delta f\|_{Y^\delta}^2= f(G_{Y}^\delta f)\quad (f \in {Y^\delta}').
$$
We conclude that for $\langle \cdot,\cdot\rangle_{Y^\delta}$ as in \eqref{16}, $u_\eps^\delta$ defined in \eqref{13} satisfies
$$
u_\eps^\delta=\argmin_{z \in X^\delta} \Big\{(Bz-g)(G_{Y}^\delta(Bz-g))
+\|Cz-h\|_{W}^2+\eps^2\|L z\|_H^2\Big\},
$$
meaning that $u_\eps^\delta \in X^\delta$ solves the Euler-Lagrange equations
\be \label{18}
\framebox{$(B u_\eps^\delta)(G_{Y}^\delta B \tilde{z})+\langle C\tilde{z},C u_\eps^\delta\rangle_W+\eps^2 \langle L\tilde{z},L u_\eps^\delta\rangle_H=
g(G_{Y}^\delta B \tilde{z})+\langle C\tilde{z},h\rangle_W \quad(\tilde{z} \in X^\delta).$}
\ee

When $G_{Y}^\delta$ can be applied in linear complexity, the iterative solution of this symmetric positive definite system can be expected to be more efficient than the iterative solution of the mixed system \eqref{14} with $\langle\cdot,\cdot\rangle_{Y^\delta}=\langle\cdot,\cdot\rangle_Y$.
So even in the case that $\langle\cdot,\cdot\rangle_Y$ can be efficiently evaluated, it can be helpful to replace $\|\cdot\|_Y$ by the above $\|\cdot\|_{Y^\delta}$ in the definition \eqref{13} of the regularized least squares approximation $u_\eps^\delta$.

\begin{remark}
We arrived at \eqref{18} without introducing the Riesz lift of $g-B u_\eps^\delta \in {Y^\delta}'$ as a separate variable. Alternatively, eliminating this variable $v^\delta_\eps$ from the mixed system \eqref{14} using the definition \eqref{16} of $\langle \cdot,\cdot\rangle_{Y^\delta}$ also results in \eqref{18}.
\end{remark}

\begin{remark}[Some scenarios where uniform preconditioners are
available]
\label{rem:available}
Continuing the discussion preceding Theorem~\ref{thm:1}, uniform preconditioners $G_{Y}^\delta$ of linear complexity are for example available when $Y$ is a Sobolev space of either positive or negative possibly non-integer smoothness index, and the $Y^\delta$ are common finite element spaces. Wavelet preconditioners can be applied, but also multi-level preconditioners that make solely use of nodal bases as the BPX preconditioner (\cite{34.3}) for positive smoothness indices, and the preconditioner from
\cite{75.258} for negative smoothness indices. Another option for negative smoothness indices is given by the preconditioner from \cite{249.985} based on the `operator preconditioning' framework where the `opposite order' operator of linear complexity uses a multi-level hierarchy.
\end{remark}

\begin{remark}[Conditioning of \eqref{18}] \label{conditioning}
In the case that $\ran (B,C,L)$ is closed, then one has $\eps^2 \lesssim \frac{(B \cdot)(G_{Y}^\delta B \cdot)+\langle C\cdot,C \cdot\rangle_W+\eps^2 \langle L \cdot,L \cdot\rangle_H}{\|\cdot\|_X^2} \lesssim 1$ on $X^\delta$ (cf. Remark~\ref{rem:closedrange}).
So for $G_X^\delta ={G_X^\delta}' \in \Lis({X^\delta}',X^\delta)$ with
$\|G_X^\delta \cdot \|_X^2 \eqsim \cdot(G_X^\delta \cdot)$ on ${X^\delta}'$, 
the condition number of the system matrix corresponding to \eqref{18} preconditioned by $G_X^\delta$ is  $\lesssim \eps^{-2}$, 
directly tying solver efficiency to regularization.
\end{remark}

\section{A posteriori residual estimation} \label{Sapost}
In view of Theorem~\ref{thm:0}, and the inequality $\|A(u-u_\eps^\delta)\|_V \leq \tau+\|f-Au_\eps^\delta\|_V$,
it is desirable to have an \emph{a posteriori} estimate for
$\|f-A u_\eps^\delta\|_V$.
Indeed, for suitable $\eps$ it will give rise to a computable upper bound for $j(u-u_\eps^\delta)$.

Considering $V=Y'\times W$, $A=(B,C)$, and $f=(g,h)$, the issue is to approximate $\|g-Bu_\eps^\delta\|_{Y'}$. This is where Fortin operators again come into play. Recall from Theorem \ref{thm:Fortin}
 that \eqref{inf-sup} guarantees the existence of a
family of uniformly bounded Fortin operators
 $(Q^\delta)_{\delta \in \Delta}$.  
We can then  write $g-Bu_\eps^\delta={Q^\delta}'(g-Bu_\eps^\delta)+(\identity -Q^\delta)' g$, and obtain
\begin{align*}
\|g-Bu_\eps^\delta\|_{Y'} &\leq\|Q^\delta\|_{\cL(Y,Y)}  \sup_{0 \neq v \in Y^\delta}\frac{|(g-Bu_\eps^\delta)(v)|}{\|v\|_Y}+\underbrace{\|(\identity -Q^\delta)' g\|_{Y'}}_{\osc^\delta(g):=}\\
&\stackrel{\eqref{equiv}}{\eqsim}
\sup_{0 \neq v \in Y^\delta}\frac{|(g-Bu_\eps^\delta)(v)|}{\|v\|_{Y^\delta}}+\osc^\delta(g)
\end{align*}

Neglecting the term $\osc^\delta(g)$, known as \emph{data oscillation}, as well as the factor $\|Q^\delta\|_{\cL(Y,Y)}$, and using that 
$ \sup_{0 \neq v \in Y^\delta}\frac{|(g-Bu_\eps^\delta)(v)|}{\|v\|_Y} \eqsim \sup_{0 \neq v \in Y^\delta}\frac{|(g-Bu_\eps^\delta)(v)|}{\|v\|_{Y^\delta}}$, we will use 
$$
\sqrt{\sup_{0 \neq v \in Y^\delta}\frac{|(g-Bu_\eps^\delta)(v)|^2}{\|v\|^2_{Y^\delta}}+\|h-Cu_\eps^\delta\|_W^2}
$$
to estimate $\|f-Au_\eps^\delta\|_V$.

Having solved $u_\eps^\delta$ from either \eqref{14} or \eqref{18}, this estimator can be computed at the expense of computing only a few inner products.
Indeed with $v^\delta_\eps \in Y^\delta$ defined by $\langle v^\delta_\eps,v\rangle_{Y^\delta}=(g-Bu_\eps^\delta)(v)$ ($v \in Y^\delta$),
the supremum under the square root equals $\| v^\delta_\eps\|_{Y^\delta}^2$.
When $u_\eps^\delta$ is determined by solving the mixed system \eqref{14}, this $v^\delta_\eps$ is the second component of the solution.
When $u_\eps^\delta$ is determined by solving the symmetric positive definite system \eqref{18}, it holds that
 $v^\delta_\eps:=G_Y^\delta(g-Bu_\eps^\delta)$ and
$\| v^\delta_\eps\|_{Y^\delta}^2=(g-Bu_\eps^\delta)(G_Y^\delta(g-Bu_\eps^\delta))$.

Since $(\identity -Q^\delta)' B X^\delta=0$, it holds that
$$
\osc^\delta(g) \leq \|Q^\delta\|_{\cL(Y,Y)} \min_{z \in X^\delta}\|g-B z\|_{Y'},
$$
so that the term that we have ignored is in any case ${\mathcal O}(\min_{z \in X^\delta}\|g-B z\|_{Y'})$. Of course this is not completely satisfactory, because it does not ensure that our estimator of $\|f-Au_\eps^\delta\|_V$ is reliable. In many cases for a suitable choice of $(Y^\delta)_{\delta \in \Delta}$,  the Fortin interpolators can be constructed such that, in any case for smooth $u$ and $g$, $\osc^\delta(g)$ is of higher order than $\min_{z \in X^\delta}\|u-z\|_{X}$, in which case it can be expected that $\osc^\delta(g)$ is asymptotically \corr{negligible}. We do not discuss this issue further.


\section{Verification of the inf-sup condition}\label{sec:inf-sup}
 
In this section we establish the validity of the inf-sup condition \eqref{inf-sup} for the examples discussed earlier.

\subsection{Verification of \eqref{inf-sup} for the Cauchy problem for Poisson's equation (Example~\ref{ex1})} \label{SverificationPoisson}
Let $\Omega \subset \R^d$ be a polytope, and recall that $A=(B_1,B_2) \in \cL\big(H^1(\Omega),(H^1_{0,\Sigma^c}(\Omega) \times \widetilde{H}^{-\frac12}(\Sigma))'\big)$  where
$(B_1z)(v)=\int_{\Omega} \nabla z \cdot \nabla v \,dx$,
and
$B_2=\gamma_\Sigma$ (see notations introduced around eq.~\eqref{17}).
By viewing $B_2$ here as an operator in $\cL\big(H^1(\Omega), (\widetilde{H}^{-\frac12}(\Sigma))'\big)$, i.e., by using that $(\widetilde{H}^{-\frac12}(\Sigma))'=H^{\frac12}(\Sigma)$, we will avoid
having to evaluate $H^{\frac12}(\Sigma)$-norms of residuals.

Let $(\tria^\delta)_{\delta \in \Delta}$ being a family of conforming, uniformly shape regular partitions of $\overline{\Omega}$ into (closed) $d$-simplices, 
with $\cF(\tria^\delta)$ denoting the set of the (closed) faces of $\tria^\delta$, and $\partial \tria^\delta:=\cup \cF(\tria^\delta)$ its skeleton.
Assuming $\overline{\Sigma}$ to be the union of some $e \in \cF(\tria^\delta)$, we take
$$
X^\delta=\cS_{\tria^\delta}^{0,1}:=\{z \in C(\overline{\Omega})\colon z|_{T} \in {\mathcal P}_1(T) \,(T \in \tria^\delta)\},
$$
but expect that similar results can be shown for higher order finite element spaces.

We will approximate the solution of the Cauchy problem by the minimizer over $X^\delta$ of the regularized least-squares functional that, in an abstract setting, was introduced and analyzed in Sect.~\ref{LS}.
To show the inf-sup condition \eqref{inf-sup} for the triple $B$, $(X^\delta)_{\delta \in \Delta}$, and a suitable family of test spaces $(Y_1^\delta \times Y_2^\delta)_{\delta \in \Delta} \subset Y_1 \times Y_2:=
H^1_{0,\Sigma^c}(\Omega) \times \widetilde{H}^{-\frac12}(\Sigma)$, it suffices to show such inf-sup conditions for
$(B_i,(X^\delta)_{\delta \in \Delta},(Y_i^\delta)_{\delta \in \Delta})$ and $i \in \{1,2\}$ separately, which will be done in Propositions~\ref{prop1}-\ref{prop2}. Indeed, uniformly bounded $Q_i^\delta \in \cL(Y_i,Y_i^\delta)$ with $(B_i X^\delta)((\identity-Q_i^\delta)Y_i)=0$ give 
uniformly bounded $Q^\delta:=Q_2^\delta\times Q_2^\delta \in \cL(Y_1\times Y_2,Y_1^\delta \times Y_2^\delta)$ with $(B X^\delta)((\identity-Q^\delta)(Y_1 \times Y_2))=0$.

\begin{proposition} \label{prop1} For each $\delta \in \Delta$, let $\tria_s^\delta$ be a refinement of $\tria^\delta$ such that\footnote{Obviously, $\tria_s^\delta$ can be a further refinement of a partition that satisfies the listed requirements. Indeed, such a further refinement renders a $Y_1^\delta$ that is only larger.\label{firstfootnote}} each $T \in \tria^\delta$ is subdivided into a uniformly bounded number of uniformly shape regular $d$-simplices, and, when $d>1$, $\tria_s^\delta$ has a vertex interior to each $e \in \cF(\tria^\delta)$ with $e \not\subset \overline{\Sigma}$.
Then for
 $$
 Y_1^\delta:=\cS_{\tria_s^\delta}^{0,1} \cap H^1_{0,\Sigma^c}(\Omega),
 $$
 it holds that
 $$
 \inf_{\delta \in \Delta} \inf_{\{z \in X^\delta\colon B_1 z\neq 0\}}\sup_{\{0 \neq v \in Y_1^\delta\}}
 \frac{|(B_1 z)(v)|}{\|B_1 z\|_{H^1_{0,\Sigma^c}(\Omega)'}\|v\|_{H^1(\Omega)}} >0.
 $$
 \end{proposition}
 
 \begin{proof}  For $d>1$, let $J^\delta$ denote a Scott-Zhang quasi-interpolator (cf. \cite{247.2}) mapping into $X^\delta \cap H^1_{0,\Sigma^c}(\Omega) \subset Y_1^\delta$,
 i.e., one that preserves homogeneous boundary conditions on $\Sigma^c$. Using the technique applied in \cite{20.05}, for $v \in H^1_{0,\Sigma^c}(\Omega)$ we define
 $$
  v^\delta=Q_{1}^\delta v:=J^\delta v+\sum_{\{e \in \cF(\tria^\delta)\colon e \not\subset \overline{\Sigma}\}} \frac{\int_e(\identity-J^\delta)v\,ds}{\int_e\phi_e ds} \phi_e \,\, \in Y_1^\delta,
$$
 where $\phi_e \in Y_1^\delta$ is such that $\phi_e$ has values in $[0,1]$, $\phi_e$ is $1$ at a vertex of $\tria_s^\delta$ interior to $e$, and $\phi_e$ vanishes outside $\cup \omega_e$, where $\omega_e:=\{T \in \tria^\delta \colon e \subset \partial T\}$.
 Then for each $e \in \cF(\tria^\delta)$, it holds that $\int_e v-v^\delta\,ds=0$, and so for $z \in X^\delta$,
 $$
(B_1 z)(v-v^\delta)= 
 \sum_{T \in \tria^\delta} \big\{-\int_T\triangle z (v-v^\delta)\,dx+ \int_{\partial T} \tfrac{\partial z}{\partial n} (v-v^\delta)\,ds\big\}=0.
 $$
 So $Q_{1}^\delta$ is a Fortin interpolator into $Y_1^\delta$. In view of Theorem~\ref{thm:Fortin} we need to show $\sup_{\delta \in \Delta} \|Q_{1}^\delta\|_{\cL(H^1_{0,\Sigma^c}(\Omega),H^1_{0,\Sigma^c}(\Omega))}<\infty$. 
 
 By an application of the trace theorem and a homogeneity argument, the construction of a Scott-Zhang quasi-interpolator shows that for arbitrary $T_e \in \tria^\delta$ with $T_e \subset \omega_e$, with $h_e:=\diam(e)$ it holds that
 $$
\|(\identity-J^\delta)v\|_{L_2(e)} \lesssim h_e^{-\frac12}\|(\identity-J^\delta)v\|_{L_2(T_e)}+h_e^{\frac12}\|(\identity-J^\delta)v\|_{H^1(T_e)}
\lesssim h_e^{\frac12} |v|_{H^1(\omega_{T_e})},
 $$
 where $\omega_{T}:=\cup\{T' \in \tria^\delta\colon |T  \cap T'|>0\}$.
 From $\int_e\phi_e ds \gtrsim h_e^{d-1}$, 
 $\|\phi_e\|_{H^1(\Omega)} \lesssim h_e^{\frac{d}{2}-1}$, and $|e|^{\frac12} \lesssim h_e^{\frac{d-1}{2}}$, we conclude
 \begin{align*}
 \|Q_{1}^\delta v\|_{H^1(\Omega)} &\leq \|J^\delta v\|_{H^1(\Omega)} +\sum_{\{e \in \cF(\tria^\delta)\colon e \not\subset \overline{\Sigma}\}} \frac{|\int_e(\identity-J^\delta)v\,ds|}{|\int_e\phi_e ds|} \|\phi_e\|_{H^1(\Omega)}\\
& \lesssim  \| v\|_{H^1(\Omega)}+\sum_{\{e \in \cF(\tria^\delta)\colon e \not\subset \overline{\Sigma}\}} h_e^{1-d}h_e^{\frac{d}{2}-1} h_e^{\frac{d-1}{2}}h_e^{\frac12} |v|_{H^1(\omega_{T_e})} \lesssim  \| v\|_{H^1(\Omega)}.
 \end{align*}
 \end{proof}

\begin{proposition} \label{prop2}
For $\delta \in \Delta$, let $\cE_s^\delta$ be a red-refinement of the partition $\cE^\delta:=\{e \in \cF(\tria^\delta)\colon e \subset \overline{\Sigma}\}$
of $\Sigma$, i.e., 
in $\cE_s^\delta$ each $(d-1)$-simplex $e \in \cE^\delta$ has been split  into $2^{d-1}$ subsimplices that are similar to $e$.
Then with
$$
 Y_2^\delta=\cS^{-1,0}_{\cE_s^\delta}:=\{v \in L_2(\Sigma)\colon v|_{\tilde e} \in \cP_0(\tilde e) \,(\tilde e \in \cE_s^\delta)\},
 $$
 it holds that
 $$
 \inf_{\delta \in \Delta} \inf_{\{z \in X^\delta\colon \gamma_\Sigma z\neq 0\}}\sup_{\{0 \neq v \in Y_2^\delta\}}
 \frac{|\int_\Sigma  z  v \,ds|}{\|\gamma_\Sigma z\|_{H^{\frac12}(\Sigma)}\|v\|_{\widetilde{H}^{-\frac12}(\Sigma)}} >0.
 $$
 \end{proposition}
 
 \begin{proof} Although formulated differently, a proof of this statement can be found in \cite[Thm.~4.1 \& Lem.~5.6]{249.975}. We summarize the main steps.
 With $\Phi^\delta=\{\phi_\nu\}$ denoting the nodal basis for $\cS^{0,1}_{\cE^\delta}$, there exists a collection $\Psi^\delta=\{\psi_\nu\} \subset  Y_2^\delta$
 for which $\langle \phi_\nu,\psi_{\nu'}\rangle_{L_2(\Sigma)} \eqsim \delta_{\nu \nu'} \|\phi_\nu\|_{L_2(\Sigma)} \|\psi_{\nu'}\|_{L_2(\Sigma)}$ and
 $\supp \psi_{\nu} \subseteq \supp \phi_{\nu}$. The resulting biorthogonal Fortin projector $Q_2^\delta$ with $\ran Q_2^\delta = \Span \Psi^\delta$ and $\ran(\identity -Q_2^\delta)=(\cS^{0,1}_{\cE^\delta})^{\perp_{L_2(\Sigma)}}$ is uniformly bounded w.r.t.~$L_2(\Sigma)$, and, thanks to the approximation properties of $\cS^{0,1}_{\cE^\delta}$, its adjoint is uniformly bounded w.r.t.~$H^1(\Sigma)$, and therefore also w.r.t.~$H^{\frac12}(\Sigma)$. We conclude that $\sup_{\delta \in \Delta} \|Q_2^\delta\|_{\cL(\widetilde{H}^{-\frac12}(\Sigma),\widetilde{H}^{-\frac12}(\Sigma))}<\infty$ as required.
 \end{proof}
 
 \begin{remark} The proof of the above proposition hinges on the construction of a (uniformly locally supported) basis dual to $\Phi^\delta$ from the span of piecewise constants w.r.t.~some refined partition. Such a construction is not restricted to red-refinement, and obviously it applies 
 when a deeper than red-refinement is applied as with the application of the bisect(5) refinement rule for two-dimensional $\Sigma$ (see e.g., \cite[Rem.~1]{242.8155}).
 \end{remark}
 
Next, in order to avoid the evaluation of the $\widetilde{H}^{-\frac12}(\Sigma)$-norm of arguments from $Y_2^\delta$, let $G_2^\delta= {G_2^\delta}' \in \Lis({Y_2^\delta}',Y_2^\delta)$ be such that $\|G_2^\delta f \|_{\widetilde{H}^{-\frac12}(\Sigma)}^2 \eqsim f(G_2^\delta f)$ ($f \in {Y_2^\delta}'$), and whose application can be performed in linear complexity (examples in \cite{75.258,249.985}). 
We equip $Y_2^\delta$ with $((G_2^\delta)^{-1}\cdot)(\cdot))^{\frac12}\eqsim \|\cdot\|_{\widetilde{H}^{-\frac12}(\Sigma)}$.

In conclusion we have that the least squares approximation $u_\eps^\delta$ from \eqref{13} of the solution of \eqref{Cauchy} is given as the minimizer over $z\in X^\delta$ of
\be \label{20}
\sup_{0 \neq (v_1,v_2) \in Y_1^\delta\times Y_2^\delta} \hspace*{-1.5em} \frac{|\int_\Omega \nabla z\cdot\nabla v_1\,dx+\int_\Sigma \gamma_\Sigma z v_2\,ds-\big(f_I(v_{1})+\int_\Sigma f_N v_1+f_D v_2\,ds\big)
|^2}{\|v_1\|_{H^1(\Omega)}^2+ ((G_2^\delta)^{-1}v_2)(v_2)}
+\eps^2\| z\|_H^2,
\ee
where $H=L_2(\Omega)$ or $H=H^1(\Omega)$ for Case~\eqref{i} or Case~\eqref{ii},
of Example \ref{ex1} in Section \ref{Sappls}, respectively.
According to \eqref{14} this $u_\eps^\delta$ can be computed as the first component of the solution $(u_\eps^\delta,v_{\eps,1}^\delta,v_{\eps,2}^\delta) \in X^\delta \times Y_1^\delta \times Y_2^\delta $ of
\begin{align*}
\int_\Omega \nabla \tilde{z} \cdot \nabla v_{\eps,1}^\delta\,dx+\int_\Sigma \gamma_\Sigma \tilde{z} v_{\eps,2}^\delta \,ds+
\int_\Omega \nabla u_\eps^\delta \cdot \nabla \tilde{v}_1\,dx+\int_\Sigma \gamma_\Sigma u_\eps^\delta \tilde{v}_2 \,ds+\langle \tilde{v}_1,v_{\eps,1}^\delta\rangle_{H^1(\Omega)}\\
+((G_2^\delta)^{-1} \tilde{v}_2)(v_{\eps,2})
-\eps^2\langle \tilde{z},u_\eps^\delta\rangle_H=f_I(\tilde{v}_1)+\int_{\Sigma} f_N \tilde{v}_1+ f_D \tilde{v}_2 \,ds
\end{align*}
for all $(\tilde{z},\tilde{v}_1,\tilde{v}_2) \in X^\delta \times Y_1^\delta \times Y_2^\delta$, which after elimination of $v_{\eps,2}^\delta$ reads as finding
$(u_\eps^\delta,v_{\eps,1}^\delta) \in X^\delta \times Y_1^\delta$
that satisfies
\be \label{21}
 \begin{split}
 \int_\Omega \nabla \tilde{z}\cdot \nabla v_{\eps,1}^\delta\,dx+ \int_\Omega& \nabla u_\eps^\delta \cdot \nabla \tilde{v}_1\,dx+\langle \tilde{v}_1, v_{\eps,1}^\delta \rangle_{H^1(\Omega)}
 -\int_\Sigma \gamma_\Sigma u_\eps^\delta G_2^\delta  \gamma_\Sigma \tilde{z}\,ds-
 \eps^2\langle u_\eps^\delta,\tilde{z}\rangle_H\\
 &=
 f_I(\tilde{v}_1)+\int_\Sigma f_N \tilde{v}_1- f_D G_2^\delta \gamma_\Sigma \tilde{z} \,ds \qquad ((\tilde{z},\tilde{v}_1) \in X^\delta \times Y_1^\delta).
 \end{split}
 \ee
 
For this $u_\eps^\delta$ the bound on $\nrm u-u_\eps^\delta\nrm_\eps$ from Theorem~\ref{thm:1} applies, and, with the specification of $j$ and $\eta$ according to the estimates in Example~\ref{ex1} \eqref{i} or \eqref{ii}, so does Theorem~\ref{thm:0}.

 \begin{remark} For completeness, we recall that in order to enhance the efficiency of an iterative solution process, additionally in \eqref{20} one may replace $\|v_1\|_{H^1(\Omega)}^2$ by an equivalent expression $((G_1^\delta)^{-1}v_1)(v_1)$, where $G_1^\delta= {G_1^\delta}' \in \Lis({Y_1^\delta}',Y_1^\delta)$ is such that $\|G_1^\delta f \|_{H^1(\Omega)}^2 \eqsim f(G_1^\delta f)$ ($f \in {Y_1^\delta}'$), and whose application can be performed in linear complexity (e.g., a multigrid preconditioner).
With this change, $\langle \tilde{v}_1,v_{\eps,1}^\delta\rangle_{H^1(\Omega)}$ in \eqref{21} reads as 
$((G_1^\delta)^{-1} \tilde{v}_1)(v_{\eps,1})$, and the unknown $v_{\eps,1}^\delta$ can be eliminated which results in the symmetric positive definite system of finding
 $u_\eps^\delta \in X^\delta$ that satisfies
 \begin{align*}
 (B_1  u_\eps^\delta)(G_1^\delta B_1  \tilde{z})&+\int_\Sigma \gamma_\Sigma u_\eps^\delta G_2^\delta  \gamma_\Sigma \tilde{z}\,ds
 + \eps^2\langle u_\eps^\delta,\tilde{z}\rangle_H \\
 &=
  f_I(G_1^\delta B_1\tilde{z})+\int_\Sigma f_N G_1^\delta B_1 \tilde{z}+f_D G_2^\delta  \gamma_\Sigma \tilde{z} \,ds \quad(\tilde z \in X^\delta),
\end{align*}
cf.~\eqref{18}.
 \end{remark}

\begin{remark}[Other approaches] In \cite{35.929}, the Dirichlet boundary condition was enforced by the so-called Nitsche method, which avoids the treatment of fractional Sobolev norms. 
Under the additional regularity condition that $u \in H^2(\Omega)$, an error estimate was derived based on the conditional stability estimates \eqref{i} or \eqref{ii}. 

In \cite{23.5}, $B_2=\gamma_\Sigma$ was viewed as a map in $\cL(H^1(\Omega),L_2(\Sigma))$, i.e., the 
discrepancy between the Dirichlet data and the trace of the approximate solution was measured in the weaker $L_2(\Sigma)$-norm.
For data such that the Cauchy problem has a (unique) solution, the solution of the resulting regularized least squares problem with $X^\delta=H^1(\Omega)$ (so without discretization) was shown to converge to the exact solution for $\eps \rightarrow 0$.
\end{remark}

\subsection{Verification of \eqref{inf-sup} for the data-assimilation problem for the heat equation (Example~\ref{ex2})} \label{SverificationHeat}
Let $\Omega \subset \R^d$ be a polytope, and recall that $A=(B,\Gamma_{I \times \omega}) \in \cL(X,V)$,
where $I=(0,T)$, $X=L_2(I;H^1(\Omega)) \cap H^1(I;H^{-1}(\Omega))$ (Case \eqref{hi}) or $X=L_2(I;H_0^1(\Omega)) \cap H^1(I;H^{-1}(\Omega))$ (Case \eqref{hii}),
$V=Y' \times L_2(I \times \omega)$, $Y=L_2(I;H_0^1(\Omega))$, $\Gamma_{I \times \omega}u=u|_{I \times \omega}$, and
$$
(Bu)(v)=\int_I \int_\Omega \partial_t u \,v +\nabla_x u \cdot \nabla_x v\,dx\,dt.
$$

\begin{proposition} \label{prop3}
Let $(\tria^\delta)_{\delta \in \Delta}$ be a family of conforming, uniformly shape regular partitions of $\overline{\Omega}$ into $d$-simplices.
For each $\delta \in \Delta$, let $\tria_s^\delta$ be a refinement of $\tria^\delta$ such that each $T \in \tria^\delta$ is subdivided into a uniformly bounded number of uniformly shape regular $d$-simplices and 
\begin{enumerate}
\item \label{first} when $d>1$, $\tria_s^\delta$ has a vertex interior to each $e \in \cF(\tria^\delta)$ with $e \not\subset \partial \Omega$,
\item \label{second} ${\displaystyle \inf_{\delta \in \Delta,\,T \in \tria^\delta} \inf_{0 \neq p \in {\mathcal P}_1(T)}\sup_{\{0 \neq \tilde p \in  H^1_0(T)\colon \tilde{p}|_{\tilde{T}} \in {\mathcal P}_1(\tilde{T}) (\tilde{T} \in \tria_s^\delta,\, \tilde{T} \subset T)\}}} \frac{\langle p, \tilde{p}\rangle_{L_2(T)}}{\|p\|_{L_2(T)}\|\tilde p\|_{L_2(T)}}>0$,
\end{enumerate}
or $\tria_s^\delta$ is a further refinement of  such a  partition (cf. footnote \ref{firstfootnote}).
With $(I^\delta)_{\delta \in \Delta}$ being a partition of $\overline{I}$, let
$$
X^\delta:=X \cap  (\cS_{I^\delta}^{0,1} \otimes \cS_{\tria^\delta}^{0,1}),\quad
Y^\delta:=\cS_{I^\delta}^{-1,1} \otimes (\cS_{\tria_s^\delta}^{0,1} \cap H^1_0(\Omega)).
$$
Then
$$
 \inf_{\delta \in \Delta} \inf_{\{z \in X^\delta\colon B z\neq 0\}}\sup_{\{0 \neq v \in Y^\delta\}}
 \frac{|(B z)(v)|}{\|B z\|_{Y'}\|v\|_{Y}} >0,
 $$
 i.e., \eqref{inf-sup} is valid.
\end{proposition}

\begin{remark}
When $\tria_s^\delta$ is generated from $\tria^\delta$ by recursive  either red-refinement or newest vertex bisection, for $d \in \{1,2,3\}$ condition \eqref{second} is satisfied as soon as the refinement is sufficiently deep such that $\tria_s^\delta$ contains a (closed) $d$-simplex in the interior of each $T \in \tria^\delta$ as can be verified by a direct computation on a reference $d$-simplex (see \cite[\S5.1]{58.6}).
\end{remark}

\begin{proof} It suffices to prove the statement for $X$ that corresponds to Case \eqref{hi}, so that $X^\delta:=\cS_{I^\delta}^{0,1} \otimes \cS_{\tria^\delta}^{0,1}$.

Let $Q_{1,x}^\delta\colon H^1_0(\Omega) \rightarrow \cS_{\tria_s^\delta}^{0,1} \cap H^1_0(\Omega)$ be the Fortin operator $Q_1$ introduced in the proof of Proposition~\ref{prop1} (for $\Sigma=\partial\Omega$). It has the properties
\begin{align*}
&\sup_{\delta \in \Delta} \|Q_{1,x}^\delta\|_{\cL(H^1_0(\Omega),H^1_0(\Omega))}<\infty,\,\,
\int_e (\identity -Q_{1,x}^\delta)v\,ds=0 \,\,(v\in H^1_0(\Omega),\,e \in \cF(\tria^\delta)),\\
&\|(\identity -Q_{1,x}^\delta)v\|_{L_2(T)} \lesssim \diam(T) |v|_{H^1(\omega_T)} \,\,(v\in H^1_0(\Omega),\,T \in \tria^\delta),
\end{align*}
which also implies $\|(\identity -{Q_{1,x}^\delta}')g\|_{H^{-1}(\Omega)} \lesssim \sum_{T \in \tria^\delta} \diam(T)^2 \|g\|_{L_2(T)}^2$ ($g \in L_2(\Omega)$).

Thanks to \eqref{second}, Theorem~\ref{thm:Fortin} shows that there exists a projector $Q_{2,x}^\delta\colon L_2(\Omega) \rightarrow \{v\in \cS^{0,1}_{\tria_s^\delta}\colon v|_{\partial\tria^\delta}=0\}$, where, for $T \in \tria^\delta$, $(Q_{2,x}^\delta v)|_T$ only depends on $v|_T$, with
$$
\sup_{\delta \in \Delta} \|Q_{2,x}^\delta\|_{\cL(L_2(\Omega),L_2(\Omega))}<\infty,\quad
\ran (\identity-Q_{2,x}^\delta) \perp_{L_2(\Omega)} \cS^{-1,1}_{\tria^\delta},
$$
so that ${Q_{2,x}^\delta}'$ reproduces $\cS^{-1,1}_{\tria^\delta}$, and so
$$
\|(\identity -{Q_{2,x}^\delta}')g\|_{L_2(T)} \lesssim \diam(T) |g|_{H^1(T)} \,\,(g \in H^1(\Omega),\,T \in \tria^\delta)
$$

With $Q_{x}^\delta:=Q_{1,x}^\delta+Q_{2,x}^\delta-Q_{2,x}^\delta Q_{1,x}^\delta\colon H^1_0(\Omega) \rightarrow \cS_{\tria_s^\delta}^{0,1} \cap H^1_0(\Omega)$, and $Q_t^\delta$ being the $L_2(I)$-orthogonal projector onto $\cS_{I^\delta}^{-1,1}$, we take $Q^\delta:=Q_t^\delta \otimes Q_x^\delta$.
From
$$
\|Q_{2,x}(\identity-Q_{1,x}) v\|_{H^1(T)} \lesssim \diam(T)^{-1} \|(\identity-Q_{1,x}) v\|_{H^1(T)} \lesssim 
|v|_{H^1(\omega_T)} \,\, (T \in \tria^\delta),
$$
one infers that $\sup_{\delta \in \Delta} \|Q^\delta\|_{\cL(Y,Y)} <\infty$.

For $z \in X^\delta$ and $v \in Y$, by writing $(B z)((\identity -Q^\delta)v)=(B z)((\identity - Q_t^\delta) \otimes \identity\,v)+(B z)(Q_t^\delta \otimes (\identity - Q_x^\delta)v)$, and by realizing that the second term equals
\begin{align*}
\sum_{J \in I^\delta} \int_J \sum_{T \in \tria^\delta} \Big\{
&\int_T
\partial_t z\, Q_t^\delta \otimes (\identity -Q_{2,x}^\delta) (\identity -Q_{1,x}^\delta) v
 \,dx
+\\
&\int_{\partial T} \frac{\partial z}{\partial n} Q_t^\delta \otimes (\identity -Q_{1,x}^\delta+Q_{2,x}^\delta Q_{1,x}^\delta ) v
 \,ds
\Big\} dt
\end{align*}
one infers that both terms vanish, so that $Q^\delta$ is a valid Fortin operator.
\end{proof}

\begin{remark} For Case~\eqref{hii} and assuming $\Omega$ being convex, in \cite[Thm.~5.7]{58.6} the statement of Proposition~\ref{prop3} was proven without Condition~\eqref{second}, which upon assuming Condition~\eqref{first}, is, however, harmless.\end{remark}

\begin{remark}[Different meshes in different time-slabs] \label{time-slab}
The result of Proposition~\ref{prop3} directly extends to the situation that
\begin{align*}
X^\delta&:=\{z \in X\colon z|_{J \times T} \in {\mathcal P}_1(J)\otimes {\mathcal P}_1(T)\,(J \in I^\delta,\,T \in \tria^\delta(J))\},\\
Y^\delta&:=\{v \in Y\colon v|_{J \times \tilde{T}} \in {\mathcal P}_1(J)\otimes {\mathcal P}_1(\tilde T)\,\,(J \in I^\delta,\,\tilde T \in \tria_{s}^\delta(J))\},
\end{align*}
i.e., when the spatial meshes underlying $X^\delta$ and $Y^\delta$ possibly depend on the time interval $J \in I^\delta$.
Notice that the global continuity of functions in $X^\delta$ implies that nonconformities between spatial meshes in adjacent time-slabs result in `hanging nodes'. 
\end{remark}

\begin{remark}To end up with really general partitions of the time-space cilinder into prismatic elements,
it would be necessary to allow elements $J \times T$ and $\widehat{J} \times \widehat{T}$ with $T \cap \widehat{T}$ being a $(d-1)$-simplex and
$J \subsetneq \widehat{J}$, i.e., to allow hanging nodes on interfaces perpendicular to the plane $\{0\} \times \Omega$.
Uniform boundedness in $\cL(Y,Y)$ of Fortin interpolators, needed to prove inf-sup stability, requires essentially that they map functions that are constant in  $x$ to functions that are constant in $x$. For such general partitions, unfortunately we do not see how this can be done.

The kind of partitions for which we are able to prove inf-sup stability do not allow local (adaptive) refinements. To circumvent this problem, in the next section we consider a reformulation of the data assimilation problem for the heat equation as a first order system.
\end{remark}

Concluding, we have that the least squares approximation $u_\eps^\delta$ from \eqref{13} of the solution  of the data-assimilation problem for the heat equation is given as the minimizer 
$$
u_\eps^\delta:=\argmin_{z \in X^\delta} \Big\{ \sup_{0 \neq v \in Y^\delta} \frac{|(Bz-g)(v)|^2}{\|v\|_{L_2(I;H^1(\Omega))}^2} +\|\Gamma_{I \times \omega} z-h\|_{L_2(I \times \omega)}^2+\eps^2 \|z\|_{L_2(I \times \Omega)}^2\Big\},
$$
where in Case~\eqref{hii} the regularizing term  $\eps^2 \|z\|_{L_2(I \times \Omega)}^2$ can be omitted.
By replacing the nominator $\|v\|_{L_2(I;H^1(\Omega))}^2$ by $((G_{Y}^\delta)^{-1} v)(v)$ for some uniform preconditioner $G_{Y}^\delta={G_{Y}^\delta}' \in \Lis({Y^\delta}',Y^\delta)$, the resulting system can be reduced to a symmetric positive definite system.
For $Y^\delta$ as in tensor product setting from Proposition~\ref{prop2} or in the time-slab setting from Remark~\ref{time-slab}, 
such preconditioners are easily constructed using multi-grid preconditioners in the spatial direction.
For the resulting $u_\eps^\delta$, the bound on $\nrm u-u^\delta_\eps\nrm_\eps$ from Theorem~\ref{thm:1} applies, and so do the bounds from Theorem~\ref{thm:0} with the specification of $j$ and $\eta$ corresponding to Example~\ref{ex2}.

\subsection{Data-assimilation for the heat equation as a first order system} \label{sec:first-order-heat}
We reconsider the data assimilation problem $Au=(B,\Gamma_{I \times \omega})u=(g,h)$ from Example~\ref{ex2}. 
Writing here $u$ as $u_1$, following \cite{75.257,75.28} for the well-posed forward heat problem  we rewrite this data-assimilation problem as a first order system
for ${\bf u}=(u_1,{\bf u}_2)$, where ${\bf u}_2=-\nabla_x u_1$.
To that end restricting $(f,g)$ to $L_2(I\times\Omega) \times L_2(I\times\omega) \subsetneq L_2(I;H^{-1}(\Omega)) \times L_2(I\times\omega)$, we consider the system
\be \label{FOS}
\widetilde{A} {\bf u}:=({\bf u}_2+\nabla_x u_1,\divv {\bf u}, \Gamma_{I \times \omega} u_1)=(0,f,g),
\ee
where $\divv {\bf u}:=\partial_t u_1+\divv_x {\bf u}_2$.
With
$$
\widetilde{X}:=\big\{{\bf u}=(u_1,{\bf u}_2) \in L_2(I;H^1(\Omega)) \times L_2(I \times \Omega)^d\colon \divv {\bf u} \in L_2(I \times \Omega)\big\}
$$
equipped with the graph norm, where for Case~\eqref{hii} the space $L_2(I;H^1(\Omega))$ should be read as $L_2(I;H_0^1(\Omega))$, it holds that
$$
\widetilde{A} \in \cL\big(\widetilde{X}, \underbrace{L_2(I \times \Omega)^d\times L_2(I \times \Omega)\times L_2(I \times \omega)}_{\widetilde{V}:=}\big).
$$
We are going to derive (un)conditional stability estimates for the first order system \eqref{FOS}.

For any ${\bf z}=(z_1,{\bf z}_2) \in \widetilde{X}$ and $v \in L_2(I;H_0^1(\Omega))$, integration-by-parts gives
$$
(B z_1)(v)=\int_I \int_\Omega \partial_t z_1 v+\nabla_x z_1\cdot \nabla_x v\,dx\,dt=
\int_I \int_\Omega \divv {\bf z} \,v +({\bf z}_2+\nabla_x z_1)\cdot \nabla_x v \,dx\,dt,
$$
so that, thanks to $\max\big(\|v\|_{L_2(I \times \Omega)},\|\nabla_x v\|_{L_2(I \times \Omega)^d}\big) \leq \|v\|_{L_2(I;H^1(\Omega))}$,
\begin{align} \nonumber
\|B z_1\|_{L_2(I;H^{-1}(\Omega))}&=
\sup_{0 \neq v \in L_2(I;H_0^1(\Omega))} \frac{\int_I \int_\Omega \partial_t z_1 v+\nabla_x z_1\cdot \nabla_x v\,dx\,dt}{\|v\|_{L_2(I;H^1(\Omega))}}\\ \label{critical}
&\leq \|\divv {\bf z}\|_{L_2(I \times \Omega)}+\|{\bf z}_2+\nabla_x z_1\|_{L_2(I \times \Omega)^d},
\end{align}
and so $\|A z_1\|_V=\sqrt{\|B z_1\|^2_{L_2(I;H^{-1}(\Omega))}+\|\Gamma_{I \times \omega} z_1\|_{L_2(I \times \omega)}^2}\, \lesssim \|\tilde{A} {\bf z}\|_{\tilde V}$.
From the conditional or unconditional stability estimates for the second order formulation from \eqref{hi} and \eqref{hii}, respectively, we infer that
\begin{align} \label{himod}
&\|z_1\|_{L_2((T_1,T_2);H^1(\breve{\omega}))} \lesssim 
\big(\|\widetilde{A} {\bf z}\|_{\widetilde V}+\|z_1\|_{L_2(I \times \Omega)}\big)^{1-\sigma} \|\widetilde{A} {\bf z}\|_{\widetilde V}^\sigma
\intertext{or} \label{hiimod}
&\|z_1\|_{L_2((T_1,T);H^1(\Omega)) \cap H^1((T_1,T);H^{-1}(\Omega)) } \lesssim \|\widetilde{A} {\bf z}\|_{\widetilde{V}},
\end{align}
in  Case~\eqref{hi} and \eqref{hii}, respectively.

To conclude (un)conditional stability for the first order system formulation in both cases, it remains to check injectivity of $(\widetilde{A},\widetilde{L})$ where $\widetilde{L} \in \cL(\widetilde{X},L_2(I\times \Omega))$ is given by $\widetilde{L} {\bf z}:= z_1$ in Case~\eqref{hi}, and $\widetilde{L}:=0$ in Case~\eqref{hii}.
In Case~\eqref{hi}, $(\widetilde{A},\widetilde{L}){\bf z}=0$ implies $z_1=0$ and thus ${\bf z}_2=-\nabla_x z_1=0$; and in Case~\eqref{hii}, 
$\widetilde{A} {\bf z}=0$ implies $A z_1=0$, which gives $z_1=0$, and so ${\bf z}_2=-\nabla_x z_1=0$.

Given a finite dimensional subspace $X^\delta \subset \widetilde{X}$, the regularized least squares approximation ${\bf u}_\eps^\delta \in X^\delta$ of the solution of \eqref{FOS}  is given by
$$
\hspace*{-1.7em}\argmin_{\mbox{} \hspace*{2em}{\bf z}=(z_1,{\bf z}_2) \in X^\delta}\hspace*{-2em}
\|{\bf z}_2\!+\!\nabla_x z_1\|_{L_2(I \times \Omega)^d}^2\!+\!\|\!\divv {\bf z}\!-\!g\|^2_{L_2(I \times \Omega)}\!+\!\| \Gamma_{I \times \omega} z_1\!-\!h\|_{L_2(I \times \omega)}^2\!+\!\eps^2\|z_1\|_{L_2(I \times \Omega)}^2,
$$
where in Case~\eqref{hii} the regularizing term  $\eps^2 \|z_1\|_{L_2(I \times \Omega)}^2$ can be omitted.
The bound on $\nrm {\bf u}-{\bf u}^\delta_\eps\nrm_\eps$ from Theorem~\ref{thm:1} applies, and so do the bounds from Theorem~\ref{thm:0} with the specification of $j$ and $\eta$ corresponding to \eqref{himod} or \eqref{hiimod} in Case~\eqref{hi} and Case~\eqref{hii}, respectively.

The main \emph{advantage} of this regularized first order system least squares (FOSLS) formulation is that all components of the residual are measured in $L_2$-type norms, so that there is no need to introduce one or more of these components as independent variables, and to ensure inf-sup stability by a careful selection of `trial' and `test' spaces. As a consequence \emph{any} finite dimensional subspace $X^\delta \subset \widetilde{X}$ can be applied.
A potential \emph{disadvantage} is that to arrive at this formulation, in \eqref{critical} the norm $\|\divv {\bf u}\|_{L_2(I;H^{-1}(\Omega))}$ was estimated on the stronger norm $\|\divv {\bf u}\|_{L_2(I\times\Omega)}$, which may result in reduced convergence rates for solutions that have singularities.
Experiments reported on in \cite{75.257} for the well-posed forward heat equation show that the risk of getting very low rates is not imaginary.

\subsection{Verification of \eqref{inf-sup} for the data-assimilation problem for the wave equation (Example~\ref{ex3})} \label{SverificationWave}
Let $\Omega \subset \R^d$ be a polytope (cf. Remark~\ref{rem:polytope}), and recall that $A=(\Box,\gamma_{I \times \partial\Omega},\Gamma_{I \times \omega}) \in \cL(X,H^{-1}(I \times \Omega)\times L_2(I \times \partial\Omega)\times L_2(I \times \omega))$.
We have to verify \eqref{inf-sup} for $B=\Box \in \cL(H^1(I \times \Omega),H^{-1}(I \times \Omega))$.

For $(\tria^\delta)_{\delta \in \Delta}$ being a family of conforming, uniformly shape regular partitions of $I \times \Omega$ into $(d+1)$-simplices,  we take
$$
X^\delta:=\cS_{\tria^\delta}^{0,1}.
$$

\begin{proposition}  \label{200} For each $\delta \in \Delta$, let $\tria_s^\delta$ be a refinement of $\tria^\delta$ such that\footref{firstfootnote} each $T \in \tria^\delta$ is subdivided into a uniformly bounded number of uniformly shape regular $(d+1)$-simplices, and $\tria_s^\delta$ has a vertex interior to each $e \in \cF(\tria^\delta)$ with $e \not\subset \partial (I \times \Omega)$.
Then for
 $$
 Y^\delta:=\cS_{\tria_s^\delta}^{0,1} \cap H^1_0(I \times \Omega),
 $$
 it holds that
 $$
 \inf_{\delta \in \Delta} \inf_{\{z \in X^\delta\colon \Box z\neq 0\}}\sup_{\{0 \neq v \in Y^\delta\}}
 \frac{|(\Box z)(v)|}{\|\Box z\|_{H^{-1}(I \times \Omega)}\|v\|_{H^1(I \times \Omega)}} >0.
 $$
 \end{proposition}
 
 \noindent The proof of this proposition is similar to the proof of Proposition~\ref{prop1}. The fact that the current result concerns the wave operator on $I \times \Omega$ instead of the Laplacian on $\Omega$ does not make any difference.

Concluding, the least squares approximation $u^\delta_\eps \equiv u^\delta$ from \eqref{13} of the solution of the data-assimilation problem for the wave equation is given as the minimizer 
$$
u^\delta:=\argmin_{z \in X^\delta} \Big\{ \sup_{0 \neq v \in Y^\delta} \frac{|(\Box z-f)(v)|^2}{\|v\|_{H^1(I \times \Omega)}^2} +
\|\gamma_{I \times \partial \Omega} z-g\|_{L_2(I \times \partial\Omega)}^2+
\|\Gamma_{I \times \omega} z-h\|_{L_2(I \times \omega)}^2\Big\}.
$$
By replacing the denominator $\|v\|_{H^1(I \times \Omega)}^2$ by $((G_{Y}^\delta)^{-1} v)(v)$ for some uniform preconditioner $G_{Y}^\delta={G_{Y}^\delta}' \in \Lis({Y^\delta}',Y^\delta)$, the resulting system can be reduced to a symmetric positive definite system.
For the resulting $u^\delta$, the bound on $\nrm u-u^\delta \nrm_\eps$ from Theorem~\ref{thm:1} applies (where $\nrm \cdot \nrm_\eps$ is $\eps$-independent because $L=0$) , and so do the bounds from Theorem~\ref{thm:0} with the specification of $j$ and $\eta$ corresponding to Example~\ref{ex3}.

\section{Numerical Experiments} \label{sec:numer}
The package P1-FEM from \cite{p1fem} was adjusted to implement the problems in \textsc{Matlab}. The finite element library \verb+NGSolve+ from \cite{247.065} was used to implement the data assimilation problem for the heat equation in two dimensions.

We consider finite element spaces w.r.t.~uniformly shape regular partitions $\tria^\delta$ of $n$-dimensional bounded domains (e.g.~$\Omega$, $I \times \Omega$, or $\Gamma \subset \partial\Omega$)
into $n$-simplices, where we restrict ourselves to partitions that are quasi-uniform.
In view of the latter, we can speak of the mesh size $h_\delta$, which number raised to the power $-n$ is proportional to $\# \tria$ and thus to the dimension of the finite element space (of fixed order).

In this section the relative error with respect to some $j\colon X\to \mathbb{R}_+$ in a numerical approximation $u^\delta_\eps$ to the prescribed solution $u$ is defined as $\frac{j(u-u_\eps^\delta)}{j(u)}$.

\subsection{Cauchy problem for Poisson's equation}
For $\Omega=(0,\pi) \times (0,1)$, $\Sigma=(0,\pi) \times \{0\}$, and $\Sigma^c =\partial\Omega \setminus \overline{\Sigma}$, 
given  $f=(f_I,f_D,f_N) \in (H^1_{0,\Sigma^c}(\Omega))' \times H^{\frac12}(\Sigma) \times H^{-\frac12}(\Sigma)$ we consider 
the problem of finding $u \in H^1(\Omega)$ that solves
\be \label{eq:Cauchy-model}
-\triangle u = f_I \text{ on } \Omega,\quad u = f_D  \text{ on } \Sigma,\quad \tfrac{\partial u}{\partial n} = f_N \text{ on }\Sigma,
\ee
or, more precisely, its variational formulation $(B_1u,B_2 u)=(g_{f_I,f_N},f_D)$ from \eqref{Cauchy}.

We consider a sequence of uniform triangulations $(\tria^{\delta})_{\delta \in \Delta}$ of $\overline{\Omega}$, where 
each next  triangulation is created from its predecessor by one uniform newest vertex bisection starting from an initial  triangulation that consists of 12 triangles created from a subdivision of $\overline{\Omega}$ into 3 rectangles of size $\frac{\pi}{3} \times 1$ by cutting each of these rectangles along their diagonals.
The three interior vertices in this initial triangulation are labelled as the `newest vertices' of all 4 triangles that contain them.

Following Sect.~\ref{SverificationPoisson}, we take $X^\delta = \cS^{0,1}_{\tria^\delta}(\Omega)$, and with $\tria^\delta_s$ denoting the second successor of $\tria^\delta$ in the sequence of triangulations,
we set $Y_1^\delta= \cS^{0,1}_{\tria_s^\delta} \cap H^1_{0,\Sigma^{c}}(\Omega)$ and  $Y_2^\delta=\cS^{-1,0}_{\cE_s^\delta}$, where $\cE_s^\delta$ is the set of edges on $\bar{\Sigma}$ of $T \in \tria^\delta_s$.

Considering the conditional stability estimate from Case~\eqref{ii} in Example~\ref{ex1} (we did not test Case~\eqref{i}),
we recall that the idea behind our approach is to compute the minimizer $u_\eps^\delta$  over $X^\delta$ of the regularized least squares functional $z\mapsto \|B_1 z -g_{f_I,f_N}\|_{H^1_{0,\Sigma^c}(\Omega)'}^2+\|\gamma_\Sigma z-f_D\|_{\tilde{H}^{-\frac12}(\Sigma)'}^2+\eps^2\|z\|_{H^1(\Omega)}^2$.
To make this method feasible without compromizing its qualitative properties, we replace the suprema over $H^1_{0,\Sigma^c}(\Omega)$ and $\tilde{H}^{-\frac12}(\Sigma)$ in the dual norms in the first two terms by suprema over $Y_1^\delta$ and $Y_2^\delta$, respectively (see Proposition~\ref{prop1}-\ref{prop2}).
At the same time we replace the norms on $H^1(\Omega)$ and $\tilde{H}^{-\frac12}(\Sigma)$ in the denominators by 
$((G_{Y_1}^\delta)^{-1} \cdot)(\cdot)^{\frac12}$ and $((G_{Y_2}^\delta)^{-1} \cdot)(\cdot)^{\frac12}$
for
preconditioners $G_{Y_i}^\delta={G_{Y_i}^\delta}' \in \Lis({Y_i^\delta}',Y_i^\delta)$ with $\|G_{Y_1}^\delta f\|_{H^1(\Omega)}^2 \eqsim f(G_{Y_1}^\delta f)$ ($f \in {Y_1^\delta}'$) and 
$ \|G_{Y_2}^\delta f\|_{\tilde{H}^{-\frac12}(\Omega)}^2 \eqsim f(G_{Y_2}^\delta f)$ ($f \in {Y_2^\delta}'$). 
Then the resulting approximation $u_\eps^\delta$ can be computed as the unique solution in $X^\delta$ of the symmetric positive definite system
$$
(B_1  u_\eps^\delta-g_{f_I,f_N})(G_{Y_1}^\delta B_1  \tilde{z})+\int_\Sigma (\gamma_\Sigma u_\eps^\delta-f_D) G_{Y_2}^\delta  \gamma_\Sigma \tilde{z}\,ds
 + \eps^2\langle u_\eps^\delta,\tilde{z}\rangle_{H^1(\Omega)} =0 \,\,\,\,(\tilde z \in X^\delta).
$$
For $G_{Y_2}^\delta$ we take the (additive) multi-level preconditioner introduced in \cite{75.258}, and for
$G_{Y_1}^\delta$ we use a common (multiplicative) multi-level preconditioner.
Both preconditioners have linear computational complexity.

In all our experiments, we prescribe the solution
$$
u(x,y)=\sin x \sinh y+{\textstyle \frac19} x^2,
$$
which corresponds to (exact) data
\be \label{eq:exact-data}
{\textstyle f=(f_I,f_D,f_N)=(-\frac29,x \mapsto \frac19 x^2, x\mapsto -\sin x).}
\ee
We measure the errors of numerical solutions in the relative $L_2(\Omega)$-norm, i.e., the  $L_2(\Omega)$-norm divided by the $L_2(\Omega)$-norm of the exact solution.
\subsubsection{Unperturbed data}
For the case of unperturbed data, we consider two strategies for choosing the regularization parameter, viz., $\eps=h_\delta$, the latter being the mesh-size, and $\eps=0$, and compare the results with those obtained with the experimentally found $\eps$ that minimizes $\|u-u_\eps^\delta\|_{L_2(\Omega)}$.
Note that since $u$ is smooth, the choice $\eps=h_\delta$ satisfies the conditions in \eqref{15}.
\begin{figure}[h]
\centering
  \includegraphics[width=0.4\linewidth]{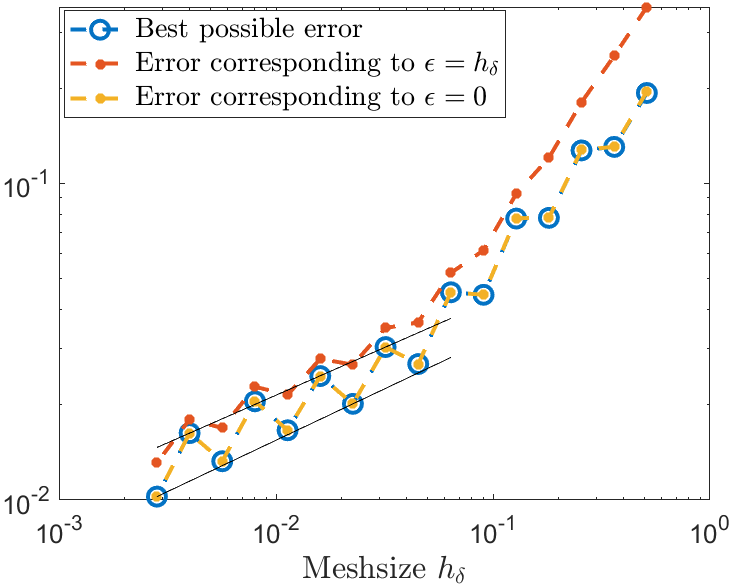}
  \caption{Poisson's equation. Meshsize vs.~relative $L_2(\Omega)$-error in case of unperturbed data.
  The asymptotic convergence rate in terms of $\#$ DoFs indicated by the solid straight lines is close to $0.15$.}
  \label{fig:unperturbed}
\end{figure}
The numerical results presented in Figure~\ref{fig:unperturbed} indicate, however, that regularization is not helpful, although it somewhat improves the conditioning of the system.

Likely the oscillations in the curves from Figure~\ref{fig:unperturbed} are due to the different geometry of the triangulations after an even or odd number of uniform refinements. The H\"{o}lder continuous behaviour, with exponent $0.15$, of the error as function of the residual, the latter being of order $h_\delta$, is better than the logarithmic dependence provided by the conditional stability estimate. That estimate, however, covers the case of a residual of most `nasty' type (and an infinitely fine mesh), whereas in our test, the residual is some specific function dependent on the partition and the prescribed solution.

\subsubsection{Randomly perturbed data}
We now perturb the Neumann datum $f_N$ with a random piecewise constant $g\in \cS^{-1,0}_{\cE_s^\delta}$ with $\|g\|_{H^{-\frac12}(\Sigma)} \eqsim \tau$.
We achieved this by normalizing a random function in $\cS^{-1,0}_{\cE_s^\delta}$, taking values in $[0,1]$,  in a discrete $H^{-\frac12}(\Sigma)$-norm, that is uniformly equivalent to the true $H^{-\frac12}(\Sigma)$-norm, and then multiplying the result with $\tau$. We used the discrete $H^{-\frac12}(\Sigma)$-norm constructed in \cite[p211]{75.258} using results from \cite{13.6}.

\begin{remark} An alternative for the latter is to construct a refinement $\cE_{ss}^\delta$ of $\cE_s^\delta$ such that for $g\in \cS^{-1,0}_{\cE_s^\delta}$, 
$\|g\|_{H^{-\frac12}(\Sigma)} \eqsim \sup_{0 \neq v \in \cS_{\cE_{ss}^\delta}^{0,1} \cap H^1_0(\Sigma)} \frac{\int_\Sigma g v\,ds}{\|v\|_{H_{00}^{\frac12}(\Sigma)}}$ (see \cite[\S3.1-2]{249.97}), after which an expression equivalent to the right-hand side can be computed using a standard multi-level preconditioner.
\end{remark}

We compare the results obtained with the regularization strategies $\eps=\tau+h_\delta$, which satisfies the conditions in \eqref{15}, and $\eps=\tau$, with those obtained with the experimentally found $\eps$ that minimizes $\|u-u_\eps^\delta\|_{L_2(\Omega)}$, which also here turns out to be $\eps=0$. The results are presented in Figure~\ref{fig:randomCauchy}.
\begin{figure}[h]
\begin{subfigure}{.5\textwidth}
\centering
\includegraphics[width=\linewidth]{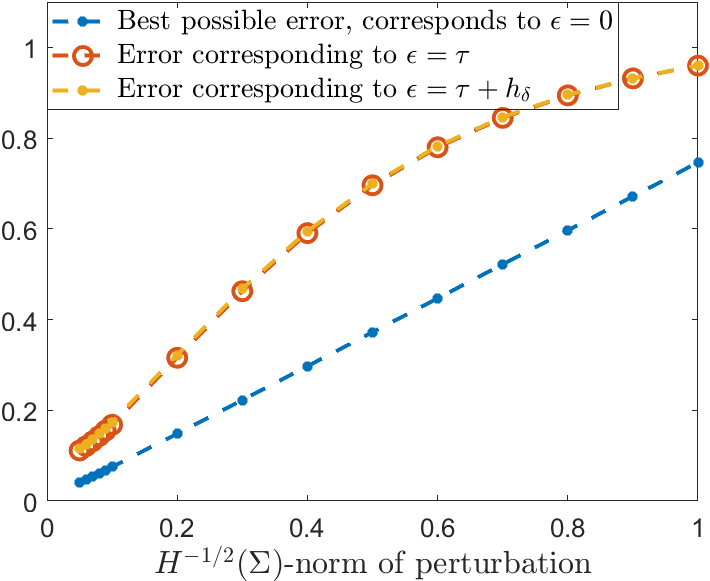}
\end{subfigure}%
\begin{subfigure}{.505\textwidth}
\centering
\includegraphics[width=\linewidth]{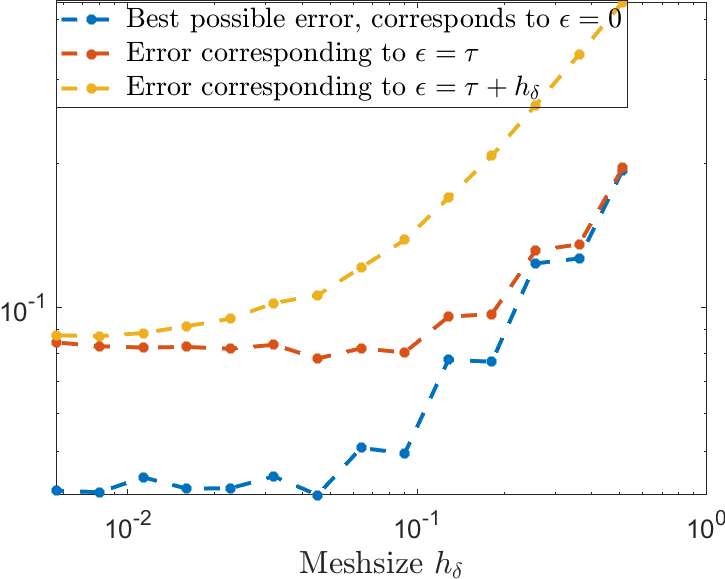}
\end{subfigure}
\caption{Poisson's equation. Norm of perturbation vs.~relative $L_2(\Omega)$-error in case of randomly perturbed Neumann data for different choices of $\eps$. Left: Random perturbations of varying $H^{-1/2}(\Sigma)$-norm and fixed mesh with $\#\mbox{DoFs} \approx 10^5$. Right: Random perturbations with $H^{-1/2}(\Sigma)$-norm equal to 0.1, and varying mesh-size.}
\label{fig:randomCauchy}
\end{figure}
We conclude that, for this problem, apparently such random perturbations are harmless, because without any regularization, for $h_\delta \downarrow 0$, which results in an increasingly ill-posed problem, their effect on the error hardly increases.

\subsubsection{`Difficult' perturbations}
From \cite{10.1} we know that for $m \in \N$, the solution $u=u^{(m)}$ of the Cauchy problem \eqref{eq:Cauchy-model} with data $f=f^{(m)}=(0,0,f_N^{(m)})$ where $f_N^{(m)}(x):= -\sqrt{\frac{2m}{\pi}}\sin m x$, is given by $u^{(m)}(x,y)=\sqrt{\frac{2}{m \pi}} \sin mx \sinh my$. It holds that $\|f_N^{(m)}\|_{H^{-\frac12}(\Sigma)}=1$, $\|u^{(m)}\|_{H^1(\Omega)} \eqsim |u^{(m)}|_{H^1(\Omega)}=\sqrt{\frac12 \sinh 2m} \sim \frac12 e^m$ ($m \rightarrow \infty$), and $\|u^{(m)}\|_{L_2(\Omega)} \sim \frac{e^m}{2\sqrt{2}m}$ ($m \rightarrow \infty$), illustrating the strong ill-posedness of the Cauchy problem. 

We investigate our numerical solver when we perturb the exact Neumann datum from \eqref{eq:exact-data} with $0.1 * f^{(m)}$. We compare the same regularization strategies as with random perturbations. The results given in Figure~\ref{fig:sinmx} show that both for $m=1, 3$ as well as for $m=16$ regularization at most slightly improves the results. For $m=1, 3$ this can be understood because the perturbation has an only modest effect on the solution. 
An explanation why for $m=16$ regularization is hardly helpful is that on the meshes that we employed apparently the best representation of 
 $u^{(16)}$ has a \emph{much} smaller norm than $u^{(16)}$ itself.
For the intermediate value $m=6$, however, we clearly see that regularization \emph{is} helpful.
\begin{figure}[h]
\begin{subfigure}{.5\textwidth}
 \centering
  \includegraphics[width=\linewidth]{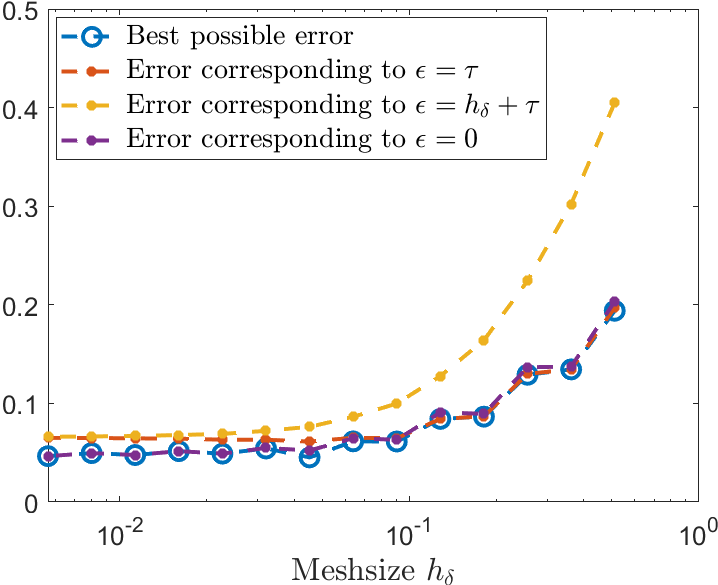}
  \caption{For $m=1$.}
  \label{fig:m1}
\end{subfigure}%
\begin{subfigure}{.5\textwidth}
 \centering
  \includegraphics[width=\linewidth]{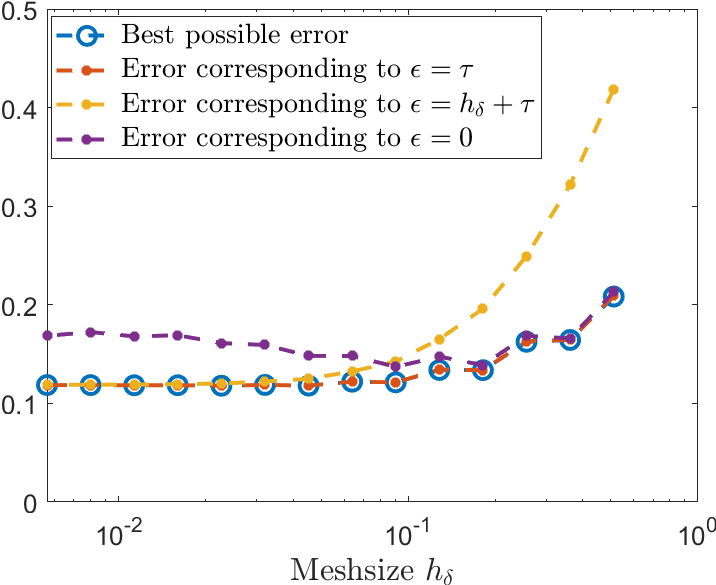}
  \caption{For $m=3$.}
  \label{fig:m3}
\end{subfigure}
\begin{subfigure}{.5\textwidth}
\centering
  \includegraphics[width=\linewidth]{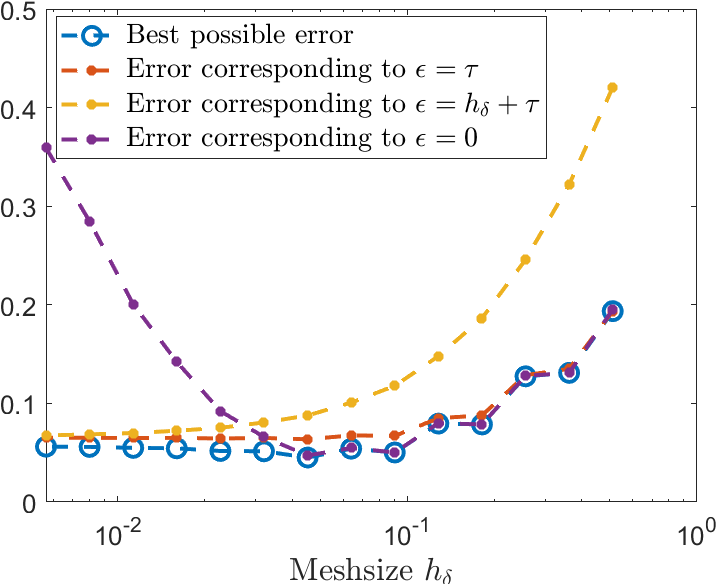}
  \caption{For $m=6$.}
  \label{fig:m6}
\end{subfigure}%
\begin{subfigure}{.5\textwidth}
\centering
  \includegraphics[width=\linewidth]{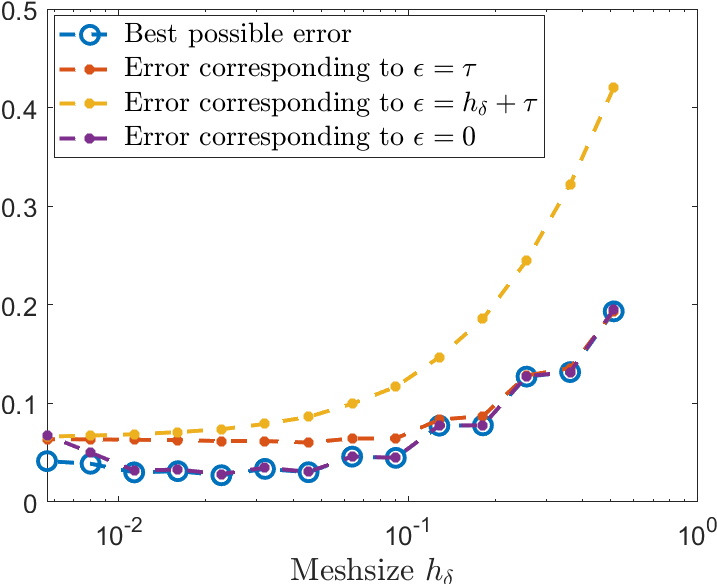}
  \caption{For $m=16$.}
  \label{fig:m16}
\end{subfigure}
\caption{Poisson's equation. Meshsize vs.~relative $L_2(\Omega)$-error for different choices of $\eps$ and perturbation with $0.1\sqrt{\frac{2m}{\pi}}\sin mx$ of the Neumann datum for $m=1,3,6,16$.}
\label{fig:sinmx}
\end{figure}


\subsection{Data-assimilation for the wave equation} \mbox{} 
For $\Omega=(0,1)$, $I=(0,1)$ and $\omega=(\tfrac12,\tfrac34)$, given $(f,g,h)\in H^{-1}(I\times\Omega)\times L_2(I\times\partial \Omega)\times L_2(I\times\omega)$, we consider the problem of finding $u$ that solves 
$$
\frac{\partial^2u}{\partial t^2} -\triangle_x u = f \text{ on } I\times\Omega,\quad u = g  \text{ on } I\times\partial\Omega,\quad u = h \text{ on }I\times\omega,
$$
or, more precisely its variational formulation $Au=(f,g,h)$ with $A:=(\Box, \gamma_{I\times \partial\Omega}, \Gamma_{I\times \omega})$ given in Example~\ref{ex3}.

We consider a sequence of uniform triangulations $(\tria^{\delta})_{\delta \in \Delta}$ of $\overline{I\times\Omega}$, where 
each next triangulation is created from its predecessor by one uniform newest vertex bisection starting from an initial  triangulation that is created by cutting $I\times\Omega$ along both diagonals. The interior vertex in this initial triangulation is labelled as the `newest vertex' of all four triangles.

Following Sect.~\ref{SverificationWave}, we take $X^\delta:=\cS_{\tria^\delta}^{0,1}(I\times\Omega)$, and with $\tria^\delta_s$ denoting the second successor of $\tria^\delta$ in the sequence of triangulations, we set $Y^\delta:= \cS^{0,1}_{\tria_s^\delta} \cap H^1_{0}(I\times\Omega)$.

Considering the \emph{un}conditional stability estimate ~\eqref{22} in Example~\ref{ex3},
our approach is to minimize the least squares functional $\|\Box z-f\|_{H^{-1}(I \times \Omega)}^2 +
\|\gamma_{I \times \partial \Omega} z-g\|_{L_2(I \times\partial\Omega)}^2+
\|\Gamma_{I \times \omega} z-h\|_{L_2(I \times \omega)}^2$ over $z \in X^\delta$, so without regularization term.
To make this method feasible without comprimizing its qualitative properties, first we replace the supremum from the first term by the supremum over $Y^\delta$ (see Proposition~\ref{200}). Second, to make the computation of the resulting dual norm efficient, we introduce a
preconditioner $G_{Y}^\delta \in \Lis({Y^\delta}',Y^\delta)$ with $\|G_{Y}^\delta f\|_{H^1(\Omega)}^2 \eqsim f(G_{Y}^\delta f)$ ($f \in {Y^\delta}'$), 
and compute our approximation $u^\delta$ as the unique solution in $X^\delta$ of the symmetric positive definite system
\begin{align*}
(\Box  u^\delta-f)(G_{Y}^\delta \Box \tilde{z})&+\langle\gamma_{I \times \partial \Omega}u^\delta-g,\gamma_{I \times \partial \Omega}\tilde{z}\rangle_{L_2(I\times\partial\Omega)} \\&+ \langle \Gamma_{I \times \omega}u^\delta-h,\Gamma_{I \times \omega}\tilde{z}\rangle_{L_2(I\times\omega)} =0 \,\,\,\,(\tilde z \in X^\delta).
\end{align*}
For $G_{Y}^\delta$ we use a common (multiplicative) multi-level preconditioner.

In our experiments, we prescribe the solution
$$
u(t,x)= \cos(\pi t)\sin(\pi x),
$$
which corresponds to (exact) data 
$$
(f,g,h)=\big(0,0,u|_{I \times \omega}\big).
$$
We perform experiments with unperturbed and perturbed data. Instead of the error in the hard to evaluate norm $\|\cdot\|_{L_\infty(I;L_2(\Omega))}+\|\partial_t\cdot\|_{L_2(I;H^{-1}(\Omega))}$ from the unconditional stability estimate \eqref{22}, we provide the a posteriori residual estimator from Section \ref{Sapost} given by
$$
\sqrt{(f-\Box  u^\delta)(G_Y^\delta(f-\Box  u_\eps^\delta)+\|\gamma_{I \times \partial \Omega} u_\eps^\delta-g\|_{L_2(I \times\partial\Omega)}^2+
\|\Gamma_{I \times \omega} u_\eps^\delta-h\|_{L_2(I \times \omega)}^2},
$$
which provides, modulo a constant factor, an upper bound for the aforementioned norm of the error
up to data oscillations. We additionally provide the relative errors in the easily evaluable $L_2(I \times \Omega)$- and $H^1(I \times \Omega)$-norms (i.e., these norms divided by the corresponding norm of the exact solution).

For the perturbed case we add to $h$ either a constant perturbation with $L_2(I \times \omega)$-norm equal to $\tau$, 
or a random perturbation of the form $p=\tau \frac{\hat p}{\|\hat p\|_{L_2(I \times \omega)}}$, where $\hat p$ is a random function in $X^\delta$ with values in $[0,1]$.
We take $\tau=0.01$. 
Figure \ref{fig:waveequation} shows the numerical results. 
\begin{figure}[h!]
\begin{subfigure}{.5\textwidth}
\centering
\includegraphics[width=\linewidth]{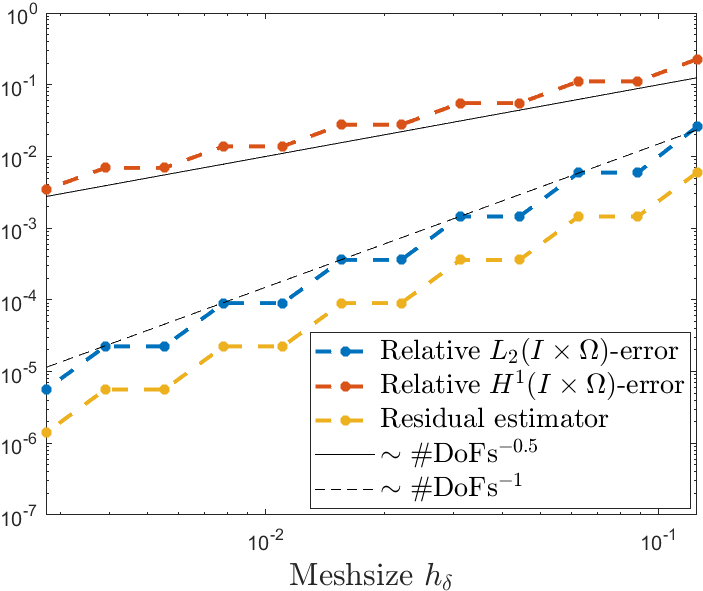}
  \caption{Unperturbed data.\\\mbox{}}
\end{subfigure}%
\begin{subfigure}{.5\textwidth}
\centering
\includegraphics[width=\linewidth]{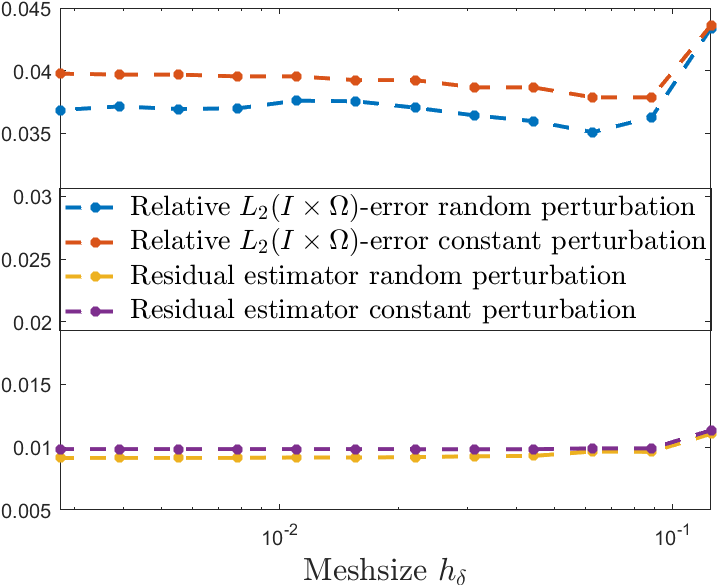}
  \caption{Perturbations, with $L_2(I \times \omega)$-norm equal to $0.01$, of datum $h$. }
\end{subfigure}
\caption{Data assimilation for the wave equation. Meshsize vs.~relative error (or residual estimator).}
\label{fig:waveequation}
\end{figure}
In the unperturbed cases, the rates for $L_2(I \times \Omega)$- and $H^1(I \times \Omega)$-norms are equal to the best approximation rates in these norms.

\subsection{Data-assimilation for the heat equation} \label{sec:numheat} \mbox{}
For $\Omega=(0,1)^d$, $I=(0,1)$ and $\omega=(\frac14,\frac34)^d$, given $(f,g)\in L_2(I;H^{-1}(\Omega)) \times L_2(I \times \omega)$,
we consider the problem of finding $u$ that solves the problem
$$
\partial_t u -\triangle_x u=f  \text{ on } I\times\Omega,\quad u|_{I \times \omega}=g,
$$
which was discussed in Example~\ref{ex2}. We consider the formulation of this problem as a first order system as analyzed in Sect.~\ref{sec:first-order-heat}. 
Assuming $f \in L_2(I \times \Omega)$, for ${\bf u}=(u_1,{\bf u}_2)=(u,-\nabla_x u)$, it reads as  $\widetilde{A} {\bf u}:=({\bf u}_2+\nabla_x u_1,\divv {\bf u}, \Gamma_{I \times \omega} u_1)=(0,f,g)$, where $\divv {\bf u}:=\partial_t u_1+\divv_x {\bf u}_2$ and ${\bf u}\in \widetilde{X}:=\big\{{\bf u}=(u_1,{\bf u}_2) \in L_2(I;H^1(\Omega)) \times L_2(I \times \Omega)^d\colon \divv {\bf u} \in L_2(I \times \Omega)\big\}$. 
Recall that we study this problem in two cases. Either we have no knowledge of $u$ on $I \times \partial\Omega$ (Case ~\eqref{hi}), or  $u$ is required to vanish on this lateral boundary (Case ~\eqref{hii}). The latter problem   is unconditionally stable.
 In this Case ~\eqref{hii}, the space $L_2(I;H^1(\Omega))$ should be read as $L_2(I;H_0^1(\Omega))$.

In view of the conditional or unconditional stability estimates \eqref{himod} or \eqref{hiimod},  respectively, given a finite dimensional subspace $X^\delta \subset \widetilde{X}$, with $\widetilde{V}:=L_2(I \times \Omega)^d \times L_2(I \times \Omega) \times L_2(I \times \omega)$ our approach is to minimize 
the least squares functional $\|\widetilde{A}{\bf u}-(0,f,g)\|_{\widetilde{V}}^2+\eps^2\|u_1\|_{L_2(I\times\Omega)}^2$ over ${\bf u} \in X^\delta$, where in Case~\eqref{hii}  the regularization term $\eps^2\|u_1\|_{L_2(I\times\Omega)}^2$ is omitted.

In our experiments, we prescribe the solution
$$
u(t,x)=(t^3+1)\prod_{i=1}^d \sin(\pi x_i),
$$
and define the data $(f,g)$ correspondingly.
For Case~\eqref{hi} the errors are measured in the relative $L_2((T_1,T_2);H^1(\breve{\omega}))$-norm, where we take $T_1=\frac18$, $T_2=\frac78$ and $\breve{\omega}=(\frac18,\frac78)^d$.
Instead of recording the error in the 
$L_2((T_1,T);H^1_0(\Omega)) \cap H^1((T_1,T);H^{-1}(\Omega))$-norm, which is
hard to evaluate, we make use of the unconditional stability estimate \eqref{hiimod} for Case~\eqref{hii}, and provide the residual $\|\widetilde{A}{\bf u}-(0,f,g)\|_{\widetilde{V}}$ which, modulo a constant factor, is an upper bound for the aforementioned norm of the error.
In addition we measure relative errors in the
$L_2((T_1,T);H^1(\Omega))$-norm for $T_1=\frac18$ which is easy to evaluate.

\subsubsection{Unperturbed data, and $\Omega=(0,1)$} \label{sec:heatunperturbed}
We consider a sequence of uniform triangulations $(\tria^{\delta})_{\delta \in \Delta}$ of $\overline{I\times\Omega}$, where 
each next triangulation is created from its predecessor by one uniform newest vertex bisection starting from an initial  triangulation that is created by cutting $I\times\Omega$ along both diagonals. The interior vertex in this initial triangulation is labelled as the `newest vertex' of all four triangles. We set $X^\delta:=\cS_{\tria^\delta}^{0,1}(I\times\Omega)\times \cS_{\tria^\delta}^{0,1}(I\times\Omega)^d$ in  Case~\eqref{hi}, and $X^\delta:=(\cS_{\tria^\delta}^{0,1}(I\times\Omega)\cap L_2(I;H_0^1(\Omega)))\times \cS_{\tria^\delta}^{0,1}(I\times\Omega)^d$ in Case~\eqref{hii}.

Taking unperturbed data, we consider two strategies for choosing the regularization parameter $\eps$ in Case~\eqref{hi}, namely $\eps=0$ and $\eps=h_\delta$, the latter being the mesh-size. Since $u$ is smooth, the choice $\eps=h_\delta$ satisfies the conditions in \eqref{15}. 
In Figure~\ref{fig:unperturbedheat}, we give the relative errors 
for both these choices of $\eps$, and also give the relative error and residual estimator in Case~\eqref{hii}. 
In the latter unconditionally stable case no regularization is applied.

As in the case of the Cauchy problem for Poisson's equation, the numerical results in Figure \ref{fig:unperturbedheat} indicate that with unperturbed data regularization is not helpful. 
The rates for $L_2((\frac18,\frac78);H^1((\frac18,\frac78)^d))$- or $L_2((\frac18,1);H^1(0,1))$-norms are equal to the best approximation rates in these norms.

\begin{figure}[h]
\centering
  \includegraphics[width=0.5\linewidth]{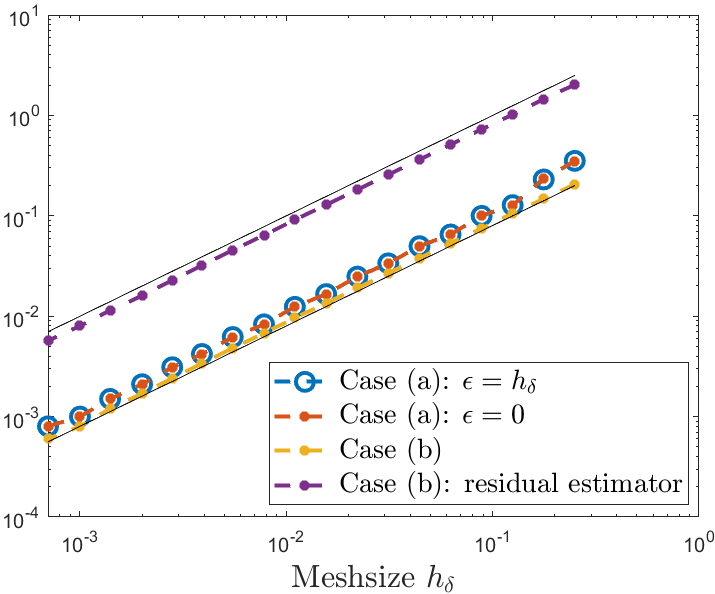}
  \caption{Data assimilation for the heat equation, and $\Omega=(0,1)$. Meshsize vs.~relative error (or residual estimator) in case of unperturbed data. Note that the norms in which the errors are measured are different for Case~\eqref{hi} and Case~\eqref{hii}. The asymptotic convergence rate in terms of $\#$ DoFs, indicated by the solid straight lines, is $0.5$.}
  \label{fig:unperturbedheat}
\end{figure}

\subsubsection{Randomly perturbed data, $\Omega=(0,1)$}
For $(X^\delta)_{\delta \in \Delta}$ as in Sect.~\ref{sec:heatunperturbed}, 
we now perturb the observational datum $g$ with a random piecewise constant $p\in \mathcal{S}^{-1,0}_{\mathcal{T}^\delta}$ with $||p||_{L_2(\omega)}=\tau$. This $p$ is constructed by normalizing a random function $\hat p\in\mathcal{S}^{-1,0}_{\mathcal{T}^\delta}$ and multiplying with $\tau$. We considered the cases where $\hat p$ takes values in either $[0,1]$ or $[-\frac12,\frac12]$.

For Case~\eqref{hi},
 we compare the results obtained with the regularization strategies $\eps=\tau$, and $\eps = \tau+h_\delta$, where the latter choice satisfies the conditions in \eqref{15}, with those obtained with the experimentally found $\eps$ that minimizes $||u-u_\eps^\delta||_{L_2(I;H^1(\Omega))}$. In addition, we present the results obtained for Case~\eqref{hii}. The results are shown in Figure \ref{fig:heatrandpert}.
\begin{figure}[h!]
\hspace*{0cm}
\begin{subfigure}{.5\textwidth}
\centering
\includegraphics[width=\linewidth]{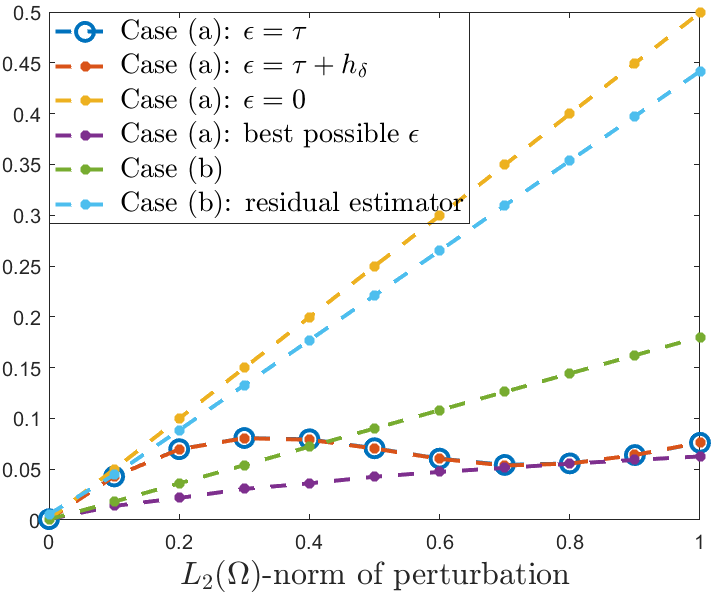}
  \label{fig:heatrandp}
\end{subfigure}\hspace*{0cm}%
\begin{subfigure}{.5\textwidth}
\centering
\includegraphics[width=\linewidth]{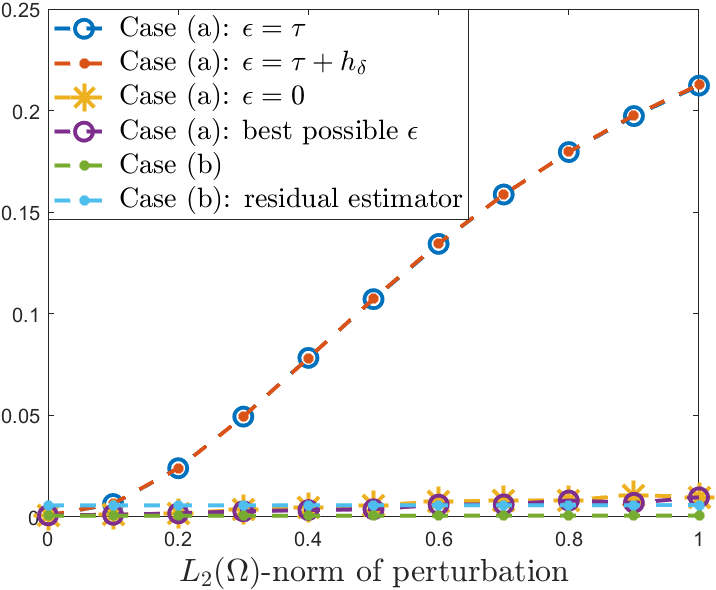}
  \label{fig:heatrandpm}
\end{subfigure}
\caption{Data assimilation for the heat equation, and $\Omega=(0,1)$. Norm of perturbation vs.~relative error (or residual estimator). Results for random perturbation of the observational datum $g$, for a fixed mesh with $\#$ DoFs $\approx10^6$ for different choices of $\eps$ and $\tau$. Left: the case where $\hat p$ takes values in $[0,1]$. Right: the case where $\hat p$ takes values in $[-1/2,1/2]$.} 
\label{fig:heatrandpert}
\end{figure}
For both Case~\eqref{hi} and \eqref{hii}, the solution is much more sensitive to random perturbations with mean $\tau/2$ than to those with mean $0$. In Case~\eqref{hi} regularization is helpful for perturbations with mean $\tau/2$, but it is not  when the mean is $0$.



\subsubsection{Unperturbed data, and $\Omega=(0,1)^2$}
We now consider the data-assimilation problem described in Sect.~\ref{sec:numheat} for unperturbed data and the two-dimensional spatial domain $\Omega=(0,1)^2$.
We consider a sequence of conforming partitions $(\tria^{\delta})_{\delta \in \Delta}$ of $\overline{I\times\Omega}$ into tetrahedra, where each partition consists of $h_\delta^{-3}$ cubes with sidelength $h_\delta$ that are decomposed into 6 tetrahedra using the Kuhn splitting.
Since with our mesh-sizes and linear finite elements we could not clearly observe convergence in Case~\eqref{hi}, we take quadratic elements, i.e., 
we set $X^\delta:=\cS_{\tria^\delta}^{0,2}(I\times\Omega)\times \cS_{\tria^\delta}^{0,2}(I\times\Omega)^2$ in  Case~\eqref{hi}, and $X^\delta:=(\cS_{\tria^\delta}^{0,2}(I\times\Omega)\cap L_2(I;H_0^1(\Omega)))\times \cS_{\tria^\delta}^{0,2}(I\times\Omega)^2$ in Case~\eqref{hii}.

We consider regularization parameters  $\eps=0$ and $\eps=h_\delta^2$ for Case~\eqref{hi}, where the latter satisfies the conditions in \eqref{15},
and for Case~\eqref{hii} apply no regularization. The results are given in Figure~\ref{fig:unperturbedheat2D}.

\begin{figure}[h]
\centering
  \includegraphics[width=0.6\linewidth]{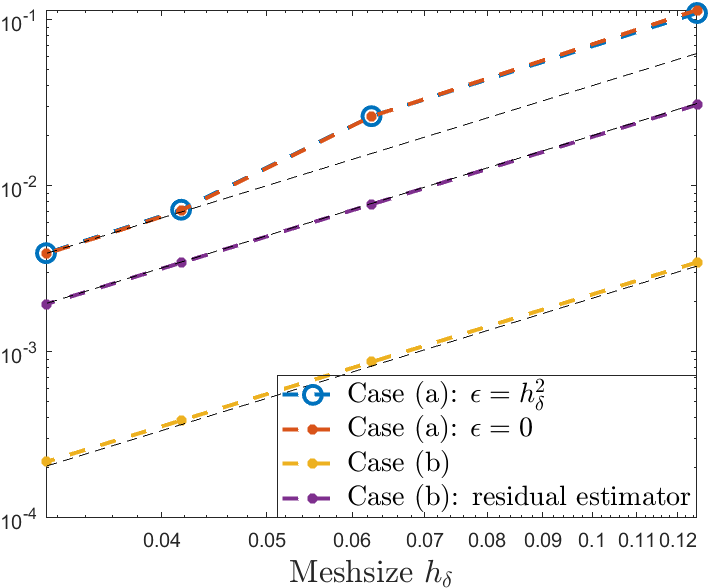}
  \caption{Data assimilation for the heat equation, and $\Omega=(0,1)^2$. Meshsize vs.~relative error (or residual estimator) in case of unperturbed data.
  The asymptotic convergence rate in terms of $\#$ DoFs indicated by the dashed straight lines is $1$.}
 
  \label{fig:unperturbedheat2D}
\end{figure}

\section{Conclusion}
We have constructed a least squares solver for general conditionally stable ill-posed PDEs.
For this solver it was demonstrated that, for a suitable regularization parameter, the error in the numerical approximation is qualitatively the best that can be expected in view of the conditional stability estimate.
In applications the least squares functional to be minimized involves negative and/or fractional Sobolev norms of residuals.
It was shown that these norms can be replaced by computable quantities without compromizing any of the attractive theoretical properties of the method.

The theoretical results were illustrated by numerical experiments for Poisson's equation with  Cauchy data, and data-assimilation problems for both heat and wave-equation. In several examples the bounds on the error in the numerical approximation that were derived using the conditional stability estimates were pessimistic, which is not surprising since these estimates cover worst case settings. Similarly, it turns out that in many cases better results were obtained by applying a smaller regularization parameter than predicted by the theoretical estimates.


\end{document}